\documentclass[11pt,a4paper]{article}
\usepackage[T1]{fontenc}
\usepackage[utf8]{inputenc}
\usepackage{lmodern,amsmath,amssymb,amsthm,mathrsfs}
\usepackage[a4paper,margin=27mm]{geometry}
\usepackage{microtype,booktabs,array}
\usepackage[compress]{cite}
\usepackage[colorlinks=true,linkcolor=blue,citecolor=blue,urlcolor=blue]{hyperref}
\hypersetup{pdftitle={Extended O-Operators on Lie–Yamaguti Algebras},pdfauthor={Wen Teng}}
\newtheorem{theorem}{Theorem}[section]
\newtheorem{proposition}[theorem]{Proposition}
\newtheorem{lemma}[theorem]{Lemma}
\newtheorem{corollary}[theorem]{Corollary}
\theoremstyle{definition}
\newtheorem{definition}[theorem]{Definition}
\newtheorem{example}[theorem]{Example}
\theoremstyle{remark}
\newtheorem{remark}[theorem]{Remark}
\numberwithin{equation}{section}
\DeclareMathOperator{\Hom}{Hom}
\DeclareMathOperator{\End}{End}
\DeclareMathOperator{\Der}{Der}
\DeclareMathOperator{\Aut}{Aut}
\DeclareMathOperator{\id}{id}
\DeclareMathOperator{\im}{im}
\DeclareMathOperator{\pr}{pr}
\DeclareMathOperator{\Gr}{Gr}
\DeclareMathOperator{\Sh}{Sh}
\DeclareMathOperator{\sgn}{sgn}
\DeclareMathOperator{\spanof}{span}
\newcommand{\K}{\mathbb K}
\newcommand{\g}{\mathfrak g}
\newcommand{\h}{\mathfrak h}
\newcommand{\cC}{\mathcal C}
\newcommand{\cyc}{\mathop{\sum\limits_{\mathrm{cyc}}}}
\newcommand{\tri}[3]{\{#1,#2,#3\}}
\newcommand{\Tri}[3]{\{#1,#2,#3\}_T}
\newcommand{\Ob}{\operatorname{Ob}}
\allowdisplaybreaks[2]
\title{Extended $\mathcal O$-operators on Lie–Yamaguti Algebras}
\author{Wen Teng\\[4pt]
\small School of Mathematics and Statistics,\\
\small Guizhou University of Finance and Economics,\\
\small Guiyang 550025, P. R. of China\\
\small Email: \href{mailto:tengwen@mail.gufe.edu.cn}{tengwen@mail.gufe.edu.cn}}
\date{}
\begin{document}
\maketitle

\begin{abstract}
We study balanced extended $\mathcal O$-operators on Lie--Yamaguti
algebras. The construction links operator equations, algebra splittings
and Yang--Baxter equations. We build a double from a representation
and a balanced equivariant map. The graphs of extended
$\mathcal O$-operators are exactly its subalgebras complementary to
the first summand. We define extended post-Lie--Yamaguti algebras
and prove their two-way relation with extended $\mathcal O$-operators.
The tensor form gives a balanced extended Lie--Yamaguti Yang--Baxter
equation. We construct a curved $L_\infty[1]$-algebra whose
Maurer--Cartan elements are extended $\mathcal O$-operators.
The normal representation of a graph gives a cohomology theory
for infinitesimal deformations, rigidity and obstructions.
We include both cyclic conditions in the lowest paired degree.
We also map operator infinitesimals to the deformation cohomology
of the induced algebra. Examples give obstructed infinitesimals,
rigid scalar operators and nontrivial formal deformations.
\end{abstract}



\section{Introduction}\label{sec:intro}

A Lie--Yamaguti algebra has a skew binary product and a ternary
product skew in its first two arguments. They satisfy four identities
\cite{Yamaguti1958,Yamaguti1969}. These identities link the binary
Jacobi sum to the ternary cyclic sum and relate the two products.
Setting the ternary product to zero gives Lie algebras. Setting the
binary product to zero gives Lie triple systems. Thus one can study
binary and ternary operators in the same setting. For further work,
see \cite{Basdouri2025,Das2026,MondalSaha,QiaoXuZhao2025,ShengZhao2022,Takahashi,WenTNEL2026,ZhaoQiaoYB2023}
and the references therein.

$\mathcal O$-operators give an operator form of the classical
Yang--Baxter equation \cite{Kupershmidt1999}. Bai, Guo and Ni
introduced extended $\mathcal O$-operators on Lie algebras. They
related them to generalized Yang--Baxter equations, double Lie
algebras and post-Lie algebras \cite{BaiGuoNi2010,BaiGuoNi2011}.
An extra equivariant map measures the defect of the usual operator
equation. Yu and Hong studied extended $\mathcal O$-operators,
post-Novikov algebras and their tensor equations \cite{YuHong2026}.
Extended $\mathcal O$-operators have also been studied on Lie triple
systems, 3-Lie algebras and perm algebras
\cite{WenTLTS2026,WenT3L2026,ZhangLin2026}. We study extended
$\mathcal O$-operators, their induced products and their tensor
equations on Lie--Yamaguti algebras.

Relative Rota--Baxter operators on Lie--Yamaguti algebras give
pre-Lie--Yamaguti algebras \cite{ShengZhao2022}. The classical
Lie--Yamaguti Yang--Baxter equation and Lie--Yamaguti bialgebras
were introduced in \cite{ZhaoQiaoYB2023}. Post-Lie--Yamaguti
algebras and relative Rota--Baxter operators of nonzero weight
were studied in \cite{QiaoXuZhao2025}. Modified Rota--Baxter
Lie--Yamaguti algebras and their deformations were studied in
\cite{Basdouri2025}. The binary and ternary products are linked
by the Lie--Yamaguti identities. Thus an extended binary operator
equation needs compatible ternary terms. We first determine these
terms and the products they induce.

We also study deformations with the algebra, representation,
extension map and mass fixed. A varying graph must stay closed
under both products of the double. We ask which infinitesimal
solutions extend to formal families and which come from
automorphisms of the fixed data. The cohomology must include all
four Lie--Yamaguti identities. In the lowest paired degree, this
requires both cyclic conditions.

Let $\g$ be a Lie--Yamaguti algebra, $(V;\rho,\mu)$ a representation
and $S:V\to\g$ a balanced equivariant map. We construct a
Lie--Yamaguti algebra $E_{\kappa,S}$ on $\g\oplus V$ using
$\K[\varepsilon]/(\varepsilon^2+\kappa)$. The construction also
works when $S$ is not injective. Extended $\mathcal O$-operators
are maps whose graphs are subalgebras of this double. Each gives
a Lie--Yamaguti structure on $V$. We define extended
post-Lie--Yamaguti algebras and construct them from extended
$\mathcal O$-operators. We also give a canonical operator for every
such algebra. The two mixed ternary products involving $T$ and $S$
are explicit. For the adjoint representation and $S=\id$, the
equations are those of \cite{Basdouri2025}.
Section~\ref{sec:operators} compares this construction with
nonzero-weight post-Lie--Yamaguti algebras.

For the coadjoint representation, we write extended
$\mathcal O$-operators in tensor form. A skew tensor gives the
operator. An admissible symmetric tensor gives its balanced
extension. The extended $\mathcal O$-operator equations are
equivalent to two extended Lie--Yamaguti Yang--Baxter equations.
At zero mass, these are the classical skew equations of
\cite{ZhaoQiaoYB2023}. On a quadratic Lie--Yamaguti algebra, the
solutions correspond to skew modified operators and maximal
isotropic graph subalgebras. We also give a complex for skew
deformations. We state separately the extra mixed ternary
conditions needed for a Lie--Yamaguti Manin triple.

We then study the Maurer--Cartan equation, cohomology and
deformations of extended $\mathcal O$-operators. For derived
brackets and relative Rota--Baxter cohomology, see
\cite{ZhaoQiao2023,ZhaoQiao2024}. For deformation maps of
Lie--Yamaguti complements, see \cite{Das2026}. Higher derived
brackets \cite{Voronov2005} give a curved $L_\infty[1]$-algebra.
Its Maurer--Cartan elements are extended $\mathcal O$-operators.
After fixing a solution, the unary bracket gives the paired
Yamaguti differential for the normal representation of its graph.
We include both cyclic coboundaries, as in \cite{Goswami2025}.
Operator obstructions satisfy both cyclic conditions. We prove
criteria for infinitesimal equivalence, rigidity and extension.
We also map operator infinitesimals to the full deformation
cohomology of the induced Lie--Yamaguti algebra.

Examples give a nonzero degenerate tensor solution, an
infinitesimal with a nonzero second-order obstruction, rigid
scalar operators and a nonconstant family with a nonzero
infinitesimal cohomology class.

Section~\ref{sec:operators} defines balanced extensions and extended
$\mathcal O$-operators, and proves the double and graph results.
Section~\ref{sec:post} defines extended post-Lie--Yamaguti algebras
and proves the two-way operator construction.
Section~\ref{sec:yang-baxter} treats tensor equations, quadratic
forms and skew deformations. Section~\ref{sec:cohomology} gives
the Maurer--Cartan equation and cohomology.
Section~\ref{sec:deformations} treats infinitesimals, rigidity,
obstructions and deformations of the induced algebra.

Throughout, $\K$ is a field of characteristic zero. Tensor
constructions use finite-dimensional spaces. The other algebraic
results also hold in infinite dimensions. Unless stated otherwise,
the algebra, representation, extension map and mass are fixed.

\section{Extended $\mathcal O$-operators and graph subalgebras}
\label{sec:operators}

We define balanced equivariant maps and extended
$\mathcal O$-operators. We construct a double and show that
extended $\mathcal O$-operators correspond to graph subalgebras
complementary to its first summand.

We first recall Lie--Yamaguti algebras and their representations.

\begin{definition}\label{def:LY}
A Lie--Yamaguti algebra is a vector space $\g$ with maps
$[\cdot,\cdot]:\wedge^2\g\to\g$ and
$\tri{\cdot}{\cdot}{\cdot}:\wedge^2\g\otimes\g\to\g$
such that, for all $x,y,z,a,b\in\g$,
\begin{gather}
 \cyc_{x,y,z}\bigl([[x,y],z]+\tri{x}{y}{z}\bigr)=0,
 \label{eq:LY1}\\
 \tri{[x,y]}{z}{a}+\tri{[y,z]}{x}{a}+\tri{[z,x]}{y}{a}=0,
 \label{eq:LY2}\\
 \tri{x}{y}{[z,a]}=[\tri{x}{y}{z},a]+[z,\tri{x}{y}{a}],
 \label{eq:LY3}\\
 \tri{x}{y}{\tri{z}{a}{b}}
 =\tri{\tri{x}{y}{z}}{a}{b}
   +\tri{z}{\tri{x}{y}{a}}{b}
   +\tri{z}{a}{\tri{x}{y}{b}}.
 \label{eq:LY4}
\end{gather}
A cyclic sum cycles the three indicated arguments together.
\end{definition}

\begin{definition}\label{def:rep}
A representation of $\g$ on $V$ is a pair of maps
$\rho:\g\to\End(V)$ and $\mu:\g\otimes\g\to\End(V)$.
Set
\begin{equation}\label{eq:D}
 D(x,y)=\mu(y,x)-\mu(x,y)+[\rho(x),\rho(y)]-\rho([x,y]).
\end{equation}
They must satisfy
\begin{align}
 \mu([x,y],z)&=\mu(x,z)\rho(y)-\mu(y,z)\rho(x),\label{eq:rep1}\\
 \mu(x,[y,z])&=\rho(y)\mu(x,z)-\rho(z)\mu(x,y),\label{eq:rep2}\\
 \rho(\tri{x}{y}{z})&=[D(x,y),\rho(z)],\label{eq:rep3}\\
 \mu(x,\tri{y}{z}{a})
 &=\mu(z,a)\mu(x,y)-\mu(y,a)\mu(x,z)+D(y,z)\mu(x,a),
 \label{eq:rep4}\\
 [D(x,y),\mu(z,a)]
 &=\mu(\tri{x}{y}{z},a)+\mu(z,\tri{x}{y}{a}).\label{eq:rep5}
\end{align}
\end{definition}

Equivalently, these products make $\g\oplus V$ a Lie--Yamaguti
algebra \cite{ShengZhao2022}:
\begin{align}
 [(x,u),(y,v)]_0&=([x,y],\rho(x)v-\rho(y)u),\label{eq:semi2}\\
 \tri{(x,u)}{(y,v)}{(z,w)}_0
 &=\bigl(\tri{x}{y}{z},\mu(y,z)u-\mu(x,z)v+D(x,y)w\bigr).
 \label{eq:semi3}
\end{align}
The adjoint representation is
\begin{equation}\label{eq:ad}
 \rho(x)u=[x,u],\qquad \mu(x,y)u=\tri{u}{x}{y},
 \qquad D(x,y)u=\tri{x}{y}{u}.
\end{equation}
The last equality uses \eqref{eq:LY1}. A cyclic identity for the
ternary product alone does not imply it.

\begin{definition}\label{def:balanced}
A linear map $S:V\to\g$ is balanced equivariant if
\begin{align}
 S\rho(x)u&=[x,Su],& S\mu(x,y)u&=\tri{Su}{x}{y},
 \label{eq:eqv}\\
 \rho(Su)v&=-\rho(Sv)u,\label{eq:bal2}\\
 \mu(Su,x)v&=-\mu(Sv,x)u,\label{eq:bal31}\\
 \mu(x,Su)v&=-D(x,Sv)u\label{eq:bal32}
\end{align}
for all $x,y\in\g$ and $u,v\in V$.
\end{definition}

The zero map is balanced equivariant. So is the inclusion of
an ideal with its restricted adjoint representation. In particular,
\eqref{eq:ad} with $S=\id$ is an example. Balance must hold in $V$.
If $S$ is not injective, checking balance only after applying it
is not enough.

\begin{lemma}\label{lem:background}
For a balanced equivariant $S$, set
\begin{equation}\label{eq:bc}
 b(u,v)=\rho(Su)v,\qquad c(u,v,w)=\mu(Sv,Sw)u.
\end{equation}
Then $(V,b,c)$ is a Lie--Yamaguti algebra and
\begin{align}
 Sb(u,v)&=[Su,Sv],& Sc(u,v,w)&=\tri{Su}{Sv}{Sw},\label{eq:bc-morphism}\\
 SD(x,y)u&=\tri{x}{y}{Su},&
 c(u,v,w)&=-\mu(Su,Sw)v=D(Su,Sv)w.\label{eq:Dlift}
\end{align}
Also, $\im S$ is an ideal of $\g$, and $\ker S$ is central in
$(V,b,c)$ in every argument of both products.
\end{lemma}

\begin{proof}
For $u,v,w\in V$, equivariance gives $Sb(u,v)=[Su,Sv]$ and
$Sc(u,v,w)=\tri{Su}{Sv}{Sw}$. By \eqref{eq:D} and \eqref{eq:LY1},
\begin{align*}
 SD(x,y)u
 &=\tri{Su}{y}{x}-\tri{Su}{x}{y}
   +[x,[y,Su]]-[y,[x,Su]]-[[x,y],Su]\\
 &=\tri{x}{y}{Su}.
\end{align*}
Balance also gives
\[
 \rho(Su)v=-\rho(Sv)u,\qquad
 \mu(Sv,Sw)u=-\mu(Su,Sw)v=D(Su,Sv)w.
\]
Thus both products have the required skew-symmetries.

We check the four Lie--Yamaguti identities on $V$.
For the first, balance moves $u$ to the module output:
\begin{align*}
 &\cyc_{u,v,w}\bigl(b(b(u,v),w)+c(u,v,w)\bigr)\\
 &=\rho(Sw)\rho(Sv)u+\rho([Sv,Sw])u-\rho(Sv)\rho(Sw)u\\
 &\quad+\mu(Sv,Sw)u+D(Sv,Sw)u-\mu(Sw,Sv)u=0.
\end{align*}
The last equality follows from the definition of $D(Sv,Sw)$.
The semidirect product for $(V;\rho,\mu)$ satisfies the second
and fourth Lie--Yamaguti identities. With only the last input in
$V$, they give
\begin{align*}
 \cyc_{x,y,z}D([x,y],z)&=0,\\
 [D(x,y),D(z,a)]
   &=D(\tri{x}{y}{z},a)+D(z,\tri{x}{y}{a}).
\end{align*}
Hence
\[
 \cyc_{u,v,w}c(b(u,v),w,a)
   =\cyc_{u,v,w}D([Su,Sv],Sw)a=0.
\]
For the third identity, \eqref{eq:rep3} implies
\begin{align*}
 c(u,v,b(w,z))
 &=D(Su,Sv)\rho(Sw)z\\
 &=\rho(\tri{Su}{Sv}{Sw})z+\rho(Sw)D(Su,Sv)z\\
 &=b(c(u,v,w),z)+b(w,c(u,v,z)).
\end{align*}
The commutator identity for $D$ gives
\begin{align*}
 c(u,v,c(w,z,a))
 &=D(Su,Sv)D(Sw,Sz)a\\
 &=D(\tri{Su}{Sv}{Sw},Sz)a
     +D(Sw,\tri{Su}{Sv}{Sz})a+D(Sw,Sz)D(Su,Sv)a\\
 &=c(c(u,v,w),z,a)+c(w,c(u,v,z),a)+c(w,z,c(u,v,a)).
\end{align*}
All four identities hold in $V$ before applying $S$.

The equalities $[x,Su]=S\rho(x)u$,
$\tri{Su}{x}{y}=S\mu(x,y)u$ and $\tri{x}{y}{Su}=SD(x,y)u$
show that $\im S$ is an ideal. If $Sk=0$, then
$b(k,u)=\rho(Sk)u=0$ and, by balance, $b(u,k)=0$.
Also, $c(k,u,v)=D(Sk,Su)v=0$, $c(u,k,v)=D(Su,Sk)v=0$
and $c(u,v,k)=\mu(Sv,Sk)u=0$. Thus $\ker S$ is central
in every argument.
\end{proof}

\begin{theorem}\label{thm:double}
For $\kappa\in\K$, the vector space $E_{\kappa,S}=\g\oplus V$
is a Lie--Yamaguti algebra with operations
\begin{align}
 [(x,u),(y,v)]_{\kappa,S}
 &=\bigl([x,y]-\kappa[Su,Sv],\rho(x)v-\rho(y)u\bigr),
 \label{eq:double2}\\
 \tri{(x,u)}{(y,v)}{(z,w)}_{\kappa,S}
 &=\Bigl(\tri{x}{y}{z}
       -\kappa\bigl(\tri{x}{Sv}{Sw}+\tri{Su}{y}{Sw}
                                +\tri{Su}{Sv}{z}\bigr),\nonumber\\
 &\hspace{2em}\mu(y,z)u-\mu(x,z)v+D(x,y)w
                  -\kappa c(u,v,w)\Bigr).\label{eq:double3}
\end{align}
The first summand $\g\oplus0$ is a subalgebra. The map
\begin{equation}\label{eq:current-map}
 (x,u)\longmapsto x\otimes1+Su\otimes\varepsilon
 \quad\text{into}\quad
 \g\otimes\K[\varepsilon]/(\varepsilon^2+\kappa)
\end{equation}
is a Lie--Yamaguti homomorphism.
\end{theorem}

\begin{proof}
Balance gives the skew-symmetries. By multilinearity, it is
enough to check the other identities on inputs in the two
summands $\g$ and $V$.

First compare the terms with values in $V$.
For the binary product of $Su,Sv$, the two lifts agree:
\[
 \rho(Su)v=-\rho(Sv)u.
\]
For a ternary product with two inputs in $\im S$, the three
placements give
\begin{align*}
 \mu(Sv,x)u&=-\mu(Su,x)v,\\
 \mu(x,Sv)u&=D(Su,x)v,\\
 -\mu(x,Su)v&=D(x,Sv)u.
\end{align*}
For three such inputs,
$\mu(Sv,Sw)u=-\mu(Su,Sw)v=D(Su,Sv)w$.
Equivariance moves $S$ through each inner product by
\eqref{eq:bc-morphism} and \eqref{eq:Dlift}. Thus a chosen
$V$-input stays in $V$ through a nested expression.
These equalities hold before applying $S$, including when
$\ker S\ne0$.

Take one of \eqref{eq:LY1}--\eqref{eq:LY4} with exactly $m$
inputs in $V$. A product with an even number of $V$-inputs
lies in $\g$. One with an odd number lies in $V$.
The scalar factor is $(-\kappa)^{\lfloor m/2\rfloor}$.
For nested products, the inner and outer factors give the same
power. Each pair of remaining odd inputs adds a factor
$-\kappa$. Thus all terms have the same factor.
This also holds in \eqref{eq:LY1}, since both types of terms
have the same three inputs.

If $m$ is even, replace each $V$-input $u$ by $Su$.
The remaining expression is an identity in $\g$, so it is zero.
If $m$ is odd, keep one $V$-input and apply $S$ to the others.
The lift equalities turn the expression into the semidirect
product identity with one $V$-input. It is again zero.
This uses neither division by $\kappa$ nor cancellation of $S$.
For example, the first identity on three $V$-inputs is
\[
 -\kappa\cyc_{u,v,w}
   \bigl(\rho([Su,Sv])w+c(u,v,w)\bigr)=0,
\]
which follows from Lemma~\ref{lem:background}.

The formulas show that the first summand is a subalgebra.
For \eqref{eq:current-map}, multiply $x+Su\varepsilon$,
$y+Sv\varepsilon$ and $z+Sw\varepsilon$ using
$\varepsilon^2=-\kappa$ and $\varepsilon^3=-\kappa\varepsilon$.
By equivariance, the coefficient of $\varepsilon$ is $S$ applied
to the second component of \eqref{eq:double2} or \eqref{eq:double3}.
The constant term is the first component. Thus the map preserves
both products.
\end{proof}

\begin{definition}\label{def:operator}
A linear map $T:V\to\g$ is a balanced extended
$\mathcal O$-operator with extension $S$ and mass $\kappa$ if
\begin{equation}
 [Tu,Tv]-T\bigl(\rho(Tu)v-\rho(Tv)u\bigr)
 =\kappa[Su,Sv].
 \label{eq:operator2}
\end{equation}
\begin{equation}
\begin{aligned}
 &\tri{Tu}{Tv}{Tw}
 -T\bigl(\mu(Tv,Tw)u-\mu(Tu,Tw)v+D(Tu,Tv)w\bigr)
 \\
 &\qquad=\kappa\Bigl(\tri{Tu}{Sv}{Sw}+\tri{Su}{Tv}{Sw}
   +\tri{Su}{Sv}{Tw}-Tc(u,v,w)\Bigr).
\end{aligned}
 \label{eq:operator3}
\end{equation}
Below, ``extended $\mathcal O$-operator'' has this meaning.
\end{definition}

\begin{theorem}\label{thm:graph}
A linear map $T$ is an extended $\mathcal O$-operator exactly when
$\Gr(T)=\{(Tu,u):u\in V\}$ is a Lie--Yamaguti subalgebra of
$E_{\kappa,S}$. Then the products
\begin{align}
 [u,v]_T&=\rho(Tu)v-\rho(Tv)u,\label{eq:desc2}\\
 \Tri{u}{v}{w}
 &=\mu(Tv,Tw)u-\mu(Tu,Tw)v+D(Tu,Tv)w-\kappa c(u,v,w)
 \label{eq:desc3}
\end{align}
make $V$ a Lie--Yamaguti algebra $V_T$.
The map $j_T:V_T\to E_{\kappa,S}$, $j_T(u)=(Tu,u)$,
is an isomorphism onto the graph.
\end{theorem}

\begin{proof}
Substitution in \eqref{eq:double2} gives
\[
 [j_T(u),j_T(v)]_{\kappa,S}
   =([Tu,Tv]-\kappa[Su,Sv],[u,v]_T).
\]
This pair lies in the graph exactly when \eqref{eq:operator2}
holds. The first component of the ternary product of three
graph elements is
\[
 \tri{Tu}{Tv}{Tw}
   -\kappa\bigl(\tri{Tu}{Sv}{Sw}+\tri{Su}{Tv}{Sw}
                                 +\tri{Su}{Sv}{Tw}\bigr),
\]
and the second is \eqref{eq:desc3}. The result lies in the graph
exactly when \eqref{eq:operator3} holds. The map $j_T$ is injective,
even if $T$ or $S$ has a kernel. Thus all four identities on $V_T$
follow from those on the graph.
\end{proof}

\begin{remark}\label{rem:comparison}
If $S=0$ or $\kappa=0$, these are the relative Rota--Baxter
equations of weight zero \cite{ShengZhao2022,ZhaoQiao2024}.
If $[\cdot,\cdot]=0$ and $\rho=0$, they give the balanced
ternary extension on a Lie triple system. For the adjoint data
$V=\g$, $S=\id$, the binary equation is the modified classical
Yang--Baxter operator equation with our sign convention.
Equation~\eqref{eq:double3} fixes the ternary equation.
Both regular equations are the modified Rota--Baxter
Lie--Yamaguti equations of \cite{Basdouri2025}, with
$\lambda=\kappa$. We study the balanced relative extension,
its associated products and its graph deformations.

The weight-one construction of \cite{QiaoXuZhao2025} uses an
action on a Lie--Yamaguti algebra $(V,b_0,c_0)$ and the graph
equations
\begin{align*}
 [Tu,Tv]&=T\bigl(\rho(Tu)v-\rho(Tv)u+b_0(u,v)\bigr),\\
 \tri{Tu}{Tv}{Tw}
 &=T\bigl(\mu(Tv,Tw)u-\mu(Tu,Tw)v+D(Tu,Tv)w+c_0(u,v,w)\bigr).
\end{align*}
This graph construction differs from ours. Such an action need
not be a balanced extension, and $\kappa=0$ need not give every
post-Lie--Yamaguti algebra. Both constructions give
pre-Lie--Yamaguti algebras when the background products vanish.
\end{remark}

\begin{proposition}\label{prop:roots}
Let $c_0\in\K$ satisfy $c_0^2=-\kappa$. For any linear $T$,
use \eqref{eq:desc2}--\eqref{eq:desc3} on $V$.
Then $T$ is an extended $\mathcal O$-operator exactly when
$T+c_0S$ preserves both products. In this case, $T-c_0S$
also preserves them. Both are homomorphisms $V_T\to\g$.
\end{proposition}

\begin{proof}
For all $u,v,w\in V$, equivariance gives
\begin{align*}
 S[u,v]_T&=[Tu,Sv]+[Su,Tv],\\
 S\Tri{u}{v}{w}
 &=\tri{Su}{Tv}{Tw}+\tri{Tu}{Sv}{Tw}
   +\tri{Tu}{Tv}{Sw}-\kappa\tri{Su}{Sv}{Sw}.
\end{align*}
For the binary product,
\begin{align*}
 &[(T+c_0S)u,(T+c_0S)v]-(T+c_0S)[u,v]_T\\
 &=[Tu,Tv]-T[u,v]_T+c_0^2[Su,Sv]\\
 &\quad+c_0\bigl([Tu,Sv]+[Su,Tv]-S[u,v]_T\bigr)\\
 &=[Tu,Tv]-T[u,v]_T-\kappa[Su,Sv].
\end{align*}
For the ternary product, expand in $c_0$ and use the second
equality above:
\begin{align*}
 &\tri{(T+c_0S)u}{(T+c_0S)v}{(T+c_0S)w}
                         -(T+c_0S)\Tri{u}{v}{w}\\
 &=\tri{Tu}{Tv}{Tw}-T\Tri{u}{v}{w}\\
 &\quad+c_0^2\bigl(\tri{Tu}{Sv}{Sw}
                 +\tri{Su}{Tv}{Sw}+\tri{Su}{Sv}{Tw}\bigr)
       +(c_0^3+\kappa c_0)\tri{Su}{Sv}{Sw}\\
 &=\tri{Tu}{Tv}{Tw}
   -T\bigl(\mu(Tv,Tw)u-\mu(Tu,Tw)v+D(Tu,Tv)w\bigr)\\
 &\quad-\kappa\bigl(\tri{Tu}{Sv}{Sw}
                 +\tri{Su}{Tv}{Sw}+\tri{Su}{Sv}{Tw}
                 -Tc(u,v,w)\bigr).
\end{align*}
Since $c_0^2=-\kappa$, the cubic coefficient is zero.
Both expressions vanish exactly when
\eqref{eq:operator2}--\eqref{eq:operator3} hold.
Then $V_T$ is a Lie--Yamaguti algebra by Theorem~\ref{thm:graph},
so a map preserving both products is a homomorphism.
The same calculation with $-c_0$ gives the result for $T-c_0S$.
\end{proof}

\begin{proposition}\label{prop:Lie}
Let $\g$ be a Lie algebra and $\rho$ a Lie algebra representation.
Let $S$ satisfy $S\rho(x)u=[x,Su]$ and $\rho(Su)v=-\rho(Sv)u$.
Give $\g$ the ternary product $\tri{x}{y}{z}=[[x,y],z]$
and set $\mu(x,y)=\rho(y)\rho(x)$. Then $S$ is balanced
equivariant. A map $T$ solves \eqref{eq:operator2} exactly when
it is an extended $\mathcal O$-operator for this Lie--Yamaguti
structure. Also,
\begin{equation}\label{eq:Lie-descent}
 \Tri{u}{v}{w}=[[u,v]_T,w]_T.
\end{equation}
\end{proposition}

\begin{proof}
The double binary bracket on $\g\ltimes_\rho V$ gives
\eqref{eq:semi2}--\eqref{eq:semi3}. Thus $(\rho,\mu)$ is a
representation and $D(x,y)=\rho([x,y])$.
Equivariance for $\mu$ follows from
$S\rho(y)\rho(x)u=[y,[x,Su]]=[[Su,x],y]$.
Applying $\rho(x)$ to binary balance gives the first ternary
balance identity. For the second,
\begin{align*}
 \rho(Su)\rho(x)v
 &=-\rho(S\rho(x)v)u=-\rho([x,Sv])u=-D(x,Sv)u.
\end{align*}
The proof of Theorem~\ref{thm:double}, using the Jacobi and
representation identities, shows that \eqref{eq:double2} is
a Lie bracket. Its double bracket is \eqref{eq:double3}.
The second component is
\[
 \rho([x,y])w+\rho(z)\rho(y)u-\rho(z)\rho(x)v
       -\kappa\rho([Su,Sv])w,
\]
and $\rho([Su,Sv])w=c(u,v,w)$ by \eqref{eq:Dlift}.
Equivariance gives the three mixed terms in the first component
of \eqref{eq:double3}. Thus binary graph closure implies
ternary closure. The converse includes binary closure by
definition. The double bracket on the graph gives
\eqref{eq:Lie-descent}.
\end{proof}

\begin{proposition}\label{prop:models}
For the adjoint data \eqref{eq:ad} with $S=\id$,
\eqref{eq:current-map} is an isomorphism. Over $\mathbb R$,
it gives
\begin{align*}
 \kappa=-c_0^2<0:&\quad (x,u)\longmapsto(x+c_0u,x-c_0u)
                                      \quad\text{in }\g\oplus\g,\\
 \kappa=c_0^2>0:&\quad (x,u)\longmapsto x+ic_0u
                                      \quad\text{in }\g_{\mathbb C},\\
 \kappa=0:&\quad (x,u)\longmapsto x+u\varepsilon,
                                      \qquad \varepsilon^2=0,
\end{align*}
where $c_0>0$ and $\g_{\mathbb C}$ is viewed as real.
In the split case, the graph of an extended $\mathcal O$-operator
is a Lie--Yamaguti complement to the diagonal. Its parametrization
is $u\mapsto((T+c_0\id)u,(T-c_0\id)u)$.
\end{proposition}

\begin{proof}
The coefficient algebra in \eqref{eq:current-map} is
$\mathbb R\oplus\mathbb R$, $\mathbb C$ or the dual numbers,
respectively. Evaluation at the two roots, at $ic_0$ or at
$\varepsilon$ gives the isomorphisms.
In the split case, equal graph components imply $2c_0u=0$.
Thus the graph meets the diagonal only at zero.
The dimension count gives a complement.
\end{proof}

\section{Extended post-Lie--Yamaguti algebras}\label{sec:post}

We define extended post-Lie--Yamaguti algebras and relate them
to extended $\mathcal O$-operators. Each extended
$\mathcal O$-operator gives such an algebra on its domain.
Conversely, each such algebra has a canonical extended
$\mathcal O$-operator. We state all identities using two copies
of the vector space, with indices for the summands.

Let $A$ have bilinear products $b$ and $*$ and trilinear
products $c,p,q,r$. Assume $b$ is skew and $c,r$ are skew in
their first two arguments. Set
\begin{align}
 B(x,y)&=x*y-y*x,\label{eq:post-B}\\
 d(x,y,z)&=p(z,y,x)-p(z,x,y)+x*(y*z)-y*(x*z)\nonumber\\
 &\quad-(x*y-y*x)*z-\kappa b(b(x,y),z),\label{eq:post-d}\\
 C(x,y,z)&=d(x,y,z)+p(x,y,z)-p(y,x,z)-\kappa c(x,y,z),
 \label{eq:post-C}\\
 h(x,y,z)&=q(x,y,z)-q(y,x,z)+r(x,y,z).\label{eq:post-h}
\end{align}
These expressions are defined by the six products.
The binary operator equation requires the term
$-\kappa b(b(x,y),z)$ in \eqref{eq:post-d}.

For indices in $\{0,1\}$, define
$b^a_{ij}:A^{\otimes2}\to A$ and $t^a_{ijk}:A^{\otimes3}\to A$
by the tables below.
\begin{equation}\label{eq:binary-table}
\begin{array}{c|cc}
 (i,j)&b^0_{ij}(x,y)&b^1_{ij}(x,y)\\\hline
 (0,0)&B(x,y)&\kappa b(x,y)\\
 (0,1)&0&x*y\\
 (1,0)&0&-y*x\\
 (1,1)&0&b(x,y)
\end{array}
\end{equation}
\begin{equation}\label{eq:ternary-table}
\begin{array}{c|cc}
 (i,j,k)&t^0_{ijk}(x,y,z)&t^1_{ijk}(x,y,z)\\\hline
 (0,0,0)&C(x,y,z)&\kappa h(x,y,z)\\
 (0,0,1)&0&d(x,y,z)\\
 (0,1,0)&0&-p(y,x,z)\\
 (1,0,0)&0&p(x,y,z)\\
 (0,1,1)&0&-q(y,x,z)\\
 (1,0,1)&0&q(x,y,z)\\
 (1,1,0)&0&r(x,y,z)\\
 (1,1,1)&0&c(x,y,z)
\end{array}
\end{equation}

\begin{definition}\label{def:post}
An extended post-Lie--Yamaguti algebra of mass $\kappa$ is
a vector space with these six products and skew-symmetries
such that, for all $i,j,k,l,m,a\in\{0,1\}$ and $x,y,z,w,v\in A$,
\begin{align}
 &\cyc_{(i,x),(j,y),(k,z)}
 \left(\sum_{e=0}^1 b^a_{ek}(b^e_{ij}(x,y),z)
                      +t^a_{ijk}(x,y,z)\right)=0,
 \label{eq:post-axiom1}\\
 &\cyc_{(i,x),(j,y),(k,z)}
    \sum_{e=0}^1 t^a_{ekl}(b^e_{ij}(x,y),z,w)=0,
 \label{eq:post-axiom2}\\
 &\sum_{e=0}^1 t^a_{ije}(x,y,b^e_{kl}(z,w))\nonumber\\
 &\qquad=\sum_{e=0}^1\left(
    b^a_{el}(t^e_{ijk}(x,y,z),w)
       +b^a_{ke}(z,t^e_{ijl}(x,y,w))\right),
 \label{eq:post-axiom3}\\
 &\sum_{e=0}^1 t^a_{ije}(x,y,t^e_{klm}(z,w,v))\nonumber\\
 &\qquad=\sum_{e=0}^1\left(
    t^a_{elm}(t^e_{ijk}(x,y,z),w,v)
       +t^a_{kem}(z,t^e_{ijl}(x,y,w),v)\right.\nonumber\\
 &\hspace{10em}\left.+t^a_{kle}(z,w,t^e_{ijm}(x,y,v))\right).
 \label{eq:post-axiom4}
\end{align}
The cyclic sums cycle the three index--argument pairs together.
A homomorphism preserves all six products and the mass.
\end{definition}

The table entries use only the six products. Thus these are
finitely many polynomial identities on $A$.
They describe a Lie--Yamaguti structure on $A\oplus A$.

\begin{theorem}\label{thm:post-double}
Definition~\ref{def:post} holds exactly when these products make
$E_A=A\oplus A$ a Lie--Yamaguti algebra:
\begin{align}
 [(x,u),(y,v)]_A
 &=\bigl(B(x,y),\kappa b(x,y)+x*v-y*u+b(u,v)\bigr),
 \label{eq:post-double2}\\
 \tri{(x,u)}{(y,v)}{(z,w)}_A
 &=\Bigl(C(x,y,z),\kappa h(x,y,z)+d(x,y,w)
      +p(u,y,z)-p(v,x,z)\nonumber\\
 &\hspace{5em}-q(v,x,w)+q(u,y,w)+r(u,v,z)+c(u,v,w)\Bigr).
 \label{eq:post-double3}
\end{align}
Then $0\oplus A$ is an ideal with products $(b,c)$.
The quotient has products $(B,C)$. We call it the subadjacent
Lie--Yamaguti algebra.
\end{theorem}

\begin{proof}
Write $x_0=(x,0)$ and $x_1=(0,x)$. Component $a$ of
$[x_i,y_j]_A$ is $b^a_{ij}(x,y)$. Component $a$ of
$\tri{x_i}{y_j}{z_k}_A$ is $t^a_{ijk}(x,y,z)$.
Thus component $a$ of $[[x_i,y_j]_A,z_k]_A$ is
$\sum_{e=0}^1 b^a_{ek}(b^e_{ij}(x,y),z)$.
The first identity is \eqref{eq:post-axiom1}.
A ternary product with an inner binary product gives
\eqref{eq:post-axiom2}--\eqref{eq:post-axiom3}.
Nested ternary products give \eqref{eq:post-axiom4}.
All placements are covered, so multilinearity proves the result.
The tables and \eqref{eq:post-B}--\eqref{eq:post-h} give the
skew-symmetries. An input in $0\oplus A$ gives zero first
component. The $11$ and $111$ entries give the ideal products.
The first components give the quotient products.
\end{proof}

For $i=j=k=0$, the second component of
\eqref{eq:post-axiom1} is
\begin{equation}\label{eq:post-expanded}
 \kappa\cyc_{x,y,z}
 \bigl(b(x*y-y*x,z)-z*b(x,y)
             +q(x,y,z)-q(y,x,z)+r(x,y,z)\bigr)=0.
\end{equation}
Keep the factor $\kappa$, including when $\kappa=0$.
On $x_0,y_0,z_1$, the same identity is
\begin{align*}
 &(x*y-y*x)*z+\kappa b(b(x,y),z)-x*(y*z)+y*(x*z)\\
 &\hspace{5em}+d(x,y,z)+p(z,x,y)-p(z,y,x)=0,
\end{align*}
which explains \eqref{eq:post-d}.

\begin{lemma}\label{lem:Peiffer}
For balanced $S$, the following operations make
$\g\oplus V$ a Lie--Yamaguti algebra:
\begin{align}
 [(x,u),(y,v)]_{\mathrm P}
 &=([x,y],\rho(x)v-\rho(y)u+b(u,v)),\label{eq:P2}\\
 \tri{(x,u)}{(y,v)}{(z,w)}_{\mathrm P}
 &=\Bigl(\tri{x}{y}{z},\mu(y,z)u-\mu(x,z)v+D(x,y)w\nonumber\\
 &\hspace{3em}-\mu(x,Sw)v+\mu(y,Sw)u
                   +\mu(Sv,z)u+c(u,v,w)\Bigr).
 \label{eq:P3}
\end{align}
Both $(x,u)\mapsto x$ and $(x,u)\mapsto x+Su$ are
homomorphisms.
\end{lemma}

\begin{proof}
The invertible map $(x,u)\mapsto(x+\frac12Su,\frac12u)$
sends these products to those of $E_{-1,S}$.
To apply its inverse to a product in $E_{-1,S}$, subtract
$S$ times the second component from the first, then double
the second. For the binary product, this gives $[x,y]$ and
$\rho(x)v-\rho(y)u+b(u,v)$. In the ternary second component,
balance reduces the terms linear in $S$ to
\[
 -\mu(x,Sw)v+\mu(y,Sw)u+\mu(Sv,z)u.
\]
Each of the four quadratic terms gives $\frac14c(u,v,w)$
by \eqref{eq:Dlift}. This gives \eqref{eq:P3}.
The identities follow from Theorem~\ref{thm:double}.
The first projection preserves both products.
Applying $S$ to each second component and adding the first
gives $[x+Su,y+Sv]$ and $\tri{x+Su}{y+Sv}{z+Sw}$.
\end{proof}

\begin{theorem}\label{thm:operator-post}
Let $T$ be an extended $\mathcal O$-operator with balanced
extension $S$ and mass $\kappa$. The operations
\begin{align}
 b(u,v)&=\rho(Su)v,&u*v&=\rho(Tu)v,\label{eq:post-induced2}\\
 c(u,v,w)&=\mu(Sv,Sw)u,&p(u,v,w)&=\mu(Tv,Tw)u,
 \label{eq:post-induced31}\\
 q(u,v,w)&=\mu(Tv,Sw)u,&r(u,v,w)&=\mu(Sv,Tw)u
 \label{eq:post-induced32}
\end{align}
make $V$ an extended post-Lie--Yamaguti algebra of mass
$\kappa$, with subadjacent algebra $V_T$. The map
\begin{equation}\label{eq:post-realization-map}
 E_V\longrightarrow\g,\qquad (u,a)\longmapsto Tu+Sa
\end{equation}
is a Lie--Yamaguti homomorphism.
\end{theorem}

\begin{proof}
Balance gives the skew-symmetries. The binary operator
equation gives
\[
 \rho([Tu,Tv])w=(u*v-v*u)*w+\kappa b(b(u,v),w).
\]
Substituting this into \eqref{eq:D} gives
$D(Tu,Tv)w=d(u,v,w)$, so $C(u,v,w)=\Tri{u}{v}{w}$.
Equivariance also gives
\begin{align}
 Sh(u,v,w)&=\tri{Tu}{Sv}{Sw}+\tri{Su}{Tv}{Sw}
                                  +\tri{Su}{Sv}{Tw},\label{eq:h-lift}\\
 [Tu,Tv]&=TB(u,v)+\kappa Sb(u,v),\label{eq:T-binary-lift}\\
 \tri{Tu}{Tv}{Tw}&=TC(u,v,w)+\kappa Sh(u,v,w).
 \label{eq:T-ternary-lift}
\end{align}
We check the defining identities without cancelling $S$.

If needed, extend the field to include $c_0$ with
$c_0^2=-\kappa$. By Proposition~\ref{prop:roots},
$T+c_0S:V_T\to\g$ is a homomorphism. Thus the pullback
\[
 \{(u,((T+c_0S)u,a)):u,a\in V\}
 \subset V_T\times(\g\oplus V)_{\mathrm P}
\]
is a Lie--Yamaguti subalgebra. Identify it with $V\oplus V$
and use the new coordinates $(u,a)\mapsto(u,a+c_0u)$.
For two elements in the first copy, the new second component is
\begin{align*}
 &-c_0\rho((T+c_0S)u)v+c_0\rho((T+c_0S)v)u
                     +c_0^2b(u,v)+c_0B(u,v)\\
 &\hspace{3em}=-c_0^2b(u,v)=\kappa b(u,v).
\end{align*}
For one element in each copy, it is
$\rho((T+c_0S)u)v-c_0b(u,v)=u*v$.
For two in the second copy, it is $b(u,v)$.
This gives the binary table.

For the ternary product, expansion and balance give
\begin{align*}
 &\mu((T+c_0S)v,(T+c_0S)w)u
  -\mu((T+c_0S)u,(T+c_0S)w)v\\
 &\qquad+D((T+c_0S)u,(T+c_0S)v)w\\
 &=\mu(Tv,Tw)u-\mu(Tu,Tw)v+D(Tu,Tv)w
       +2c_0h(u,v,w)+3c_0^2c(u,v,w),\\
 &-\mu((T+c_0S)u,Sw)v+\mu((T+c_0S)v,Sw)u
                         +\mu(Sv,(T+c_0S)w)u\\
 &=h(u,v,w)+3c_0c(u,v,w).
\end{align*}
With all inputs in the first copy, the second component is
\begin{align*}
 &-c_0\bigl(\mu(Tv,Tw)u-\mu(Tu,Tw)v+D(Tu,Tv)w
                      +2c_0h(u,v,w)+3c_0^2c(u,v,w)\bigr)\\
 &\quad+c_0^2\bigl(h(u,v,w)+3c_0c(u,v,w)\bigr)
       -c_0^3c(u,v,w)+c_0C(u,v,w)\\
 &=\kappa h(u,v,w).
\end{align*}
For the $100$ placement it is
\begin{align*}
 &\mu((T+c_0S)v,(T+c_0S)w)u
   -c_0\mu((T+c_0S)v,Sw)u\\
 &\hspace{4em}-c_0\mu(Sv,(T+c_0S)w)u
      +c_0^2\mu(Sv,Sw)u=\mu(Tv,Tw)u.
\end{align*}
The $101$ and $110$ entries are
$\mu((T+c_0S)v,Sw)u-c_0c(u,v,w)=q(u,v,w)$ and
$\mu(Sv,(T+c_0S)w)u-c_0c(u,v,w)=r(u,v,w)$.
The $111$ entry is $c(u,v,w)$. Skew-symmetry gives $010$ and $011$.
The $001$ entry is $D(Tu,Tv)w=d(u,v,w)$. It also follows from
the first Lie--Yamaguti identity on $u_0,v_0,w_1$ above.
This gives the ternary table. Theorem~\ref{thm:post-double}
proves the defining identities. Their coefficients lie in $\K$,
so the identities hold over $\K$.

Equations \eqref{eq:T-binary-lift}--\eqref{eq:T-ternary-lift}
show that \eqref{eq:post-realization-map} preserves products
with all inputs in the first copy. Equivariance handles the
other placements; for example, $Sq(u,v,w)=\tri{Su}{Tv}{Sw}$
and $Sr(u,v,w)=\tri{Su}{Sv}{Tw}$. The tables and multilinearity
give the result.
\end{proof}

\begin{theorem}\label{thm:post-operator}
Let $A$ be an extended post-Lie--Yamaguti algebra. On the
ideal $0\oplus A$ in $E_A$ define
\begin{align}
 \rho_A(X)u&=\pr_2[X,(0,u)]_A,&
 \mu_A(X,Y)u&=\pr_2\tri{(0,u)}{X}{Y}_A,\label{eq:post-action}\\
 S_Au&=(0,u),&T_Au&=(u,0).\label{eq:post-canonical}
\end{align}
Then $(A;\rho_A,\mu_A)$ is a representation of $E_A$.
The map $S_A$ is balanced equivariant, and $T_A$ is an extended
$\mathcal O$-operator of mass $\kappa$.
Formulas \eqref{eq:post-induced2}--\eqref{eq:post-induced32}
recover the six original products.
\end{theorem}

\begin{proof}
The adjoint representation of $E_A$ preserves the ideal
$0\oplus A$. Its restriction gives the five representation
identities for $(\rho_A,\mu_A)$. In particular,
\begin{align*}
 \rho_A((x,u))v&=x*v+b(u,v),\\
 \mu_A((x,u),(y,v))w
   &=p(w,x,y)+q(w,x,v)+r(w,u,y)+c(w,u,v).
\end{align*}
The map $S_A$ includes this ideal. Thus
$S_A\rho_A(X)u=[X,S_Au]_A$ and
$S_A\mu_A(X,Y)u=\tri{S_Au}{X}{Y}_A$.
Evaluating the adjoint products on two or three ideal elements
and using skew-symmetry gives the three balance conditions.
Hence $S_A$ is balanced equivariant.

By \eqref{eq:post-d} and the formula for $D_A$,
\[
 D_A(T_Au,T_Av)w=d(u,v,w).
\]
It follows that
\begin{align*}
 \rho_A(T_Au)v-\rho_A(T_Av)u&=B(u,v),\\
 \mu_A(T_Av,T_Aw)u-\mu_A(T_Au,T_Aw)v
                     +D_A(T_Au,T_Av)w
   &=C(u,v,w)+\kappa c(u,v,w).
\end{align*}
The binary operator identity is now
\[
 [T_Au,T_Av]_A-T_AB(u,v)
     =(0,\kappa b(u,v))=\kappa[S_Au,S_Av]_A.
\]
For the ternary identity, its left-hand side is
\begin{align*}
 &\tri{T_Au}{T_Av}{T_Aw}_A
   -T_A\bigl(C(u,v,w)+\kappa c(u,v,w)\bigr)\\
 &=(C(u,v,w),\kappa h(u,v,w))
           -(C(u,v,w)+\kappa c(u,v,w),0)\\
 &=(-\kappa c(u,v,w),\kappa h(u,v,w)).
\end{align*}
The three mixed brackets are
\begin{align*}
 \tri{T_Au}{S_Av}{S_Aw}_A&=(0,-q(v,u,w)),\\
 \tri{S_Au}{T_Av}{S_Aw}_A&=(0,q(u,v,w)),\\
 \tri{S_Au}{S_Av}{T_Aw}_A&=(0,r(u,v,w)).
\end{align*}
Their sum is $(0,h(u,v,w))$, and
$T_Ac(u,v,w)=(c(u,v,w),0)$. Thus the right side of the ternary
extended $\mathcal O$-operator equation is the same pair.
Hence $T_A$ is an extended $\mathcal O$-operator.
Substituting $S_Au=(0,u)$ and $T_Au=(u,0)$ in the
representation formulas recovers $b,*,c,p,q,r$.
\end{proof}

\begin{corollary}
If $b=c=q=r=0$, Definition~\ref{def:post} gives the
pre-Lie--Yamaguti structure from an ordinary
$\mathcal O$-operator, with products $*,p$.
For the Lie algebra data of Proposition~\ref{prop:Lie}, the
construction agrees with the binary extended $\mathcal O$-operator
and its induced Lie bracket.
\end{corollary}

\begin{proof}
In the first case, the double is an abelian semidirect
extension of $(A,B,C)$, and $T_A$ satisfies the ordinary operator
equations. The representation identities are the
pre-Lie--Yamaguti identities of \cite{ShengZhao2022}.
Conversely, those identities give this double.
For the second claim, apply Theorem~\ref{thm:operator-post}
to Proposition~\ref{prop:Lie} and use \eqref{eq:Lie-descent}.
\end{proof}

\section{Extended Lie--Yamaguti Yang--Baxter equation}
\label{sec:yang-baxter}

We define the balanced extended Lie--Yamaguti Yang--Baxter
equation and relate it to extended $\mathcal O$-operators for
the coadjoint representation. At zero mass, it is the classical
equation of \cite{ZhaoQiaoYB2023}. We also treat skew modified
operators on quadratic algebras, maximal isotropic graphs and
skew deformations. Here $\g$ is finite dimensional, and the
symmetric extension tensor is fixed.

Let $(\rho^\vee,\mu^\vee)$ be the coadjoint representation,
with associated map $D^\vee$. We use the signs
\begin{align}
 \langle\rho^\vee(x)\xi,z\rangle
   &=-\langle\xi,[x,z]\rangle,\nonumber\\
 \langle\mu^\vee(x,y)\xi,z\rangle
   &=\langle\xi,\tri{z}{y}{x}\rangle,\label{eq:coadjoint}\\
 \langle D^\vee(x,y)\xi,z\rangle
   &=-\langle\xi,\tri{x}{y}{z}\rangle.\nonumber
\end{align}
Thus $\mu^\vee=-\mu_{\rm ad}^*\tau$, where the star means
negative transpose and $\tau$ swaps the two algebra inputs.
This swap is required.

For $r\in\wedge^2\g$ and $s\in\operatorname{Sym}^2\g$ set
\begin{equation}\label{eq:sharp-convention}
 T=r^\sharp,\quad S=s^\sharp,\qquad
 \langle r^\sharp\xi,\eta\rangle
       =\langle r,\xi\otimes\eta\rangle,\qquad
 \langle s^\sharp\xi,\eta\rangle
       =\langle s,\xi\otimes\eta\rangle.
\end{equation}
Call $s$ admissible if $S$ satisfies
\eqref{eq:eqv}--\eqref{eq:bal32} for $(\rho^\vee,\mu^\vee)$,
including when $s$ is degenerate. Set
\[
 c_S(\xi,\eta,\zeta)
     =\mu^\vee(S\eta,S\zeta)\xi
     =D^\vee(S\xi,S\eta)\zeta.
\]

Use leg notation, with products in the common leg. Set
\begin{align}
 \mathsf Y_2(r)
  &=[r_{12},r_{13}]+[r_{12},r_{23}]+[r_{13},r_{23}],
 \label{eq:Y2}\\
 \mathsf Y_3(r)
  &=-\bigl(\tri{r_{31}}{r_{32}}{r_{43}}
          +\tri{r_{13}}{r_{41}}{r_{12}}
          +\tri{r_{42}}{r_{23}}{r_{21}}
          +\tri{r_{41}}{r_{42}}{r_{43}}\bigr).
 \label{eq:Y3}
\end{align}
Here $r_{ab}=\sum_i x_i^{(a)}y_i^{(b)}$ for
$r=\sum_i x_i\otimes y_i$. The minus sign in \eqref{eq:Y3}
matches the first-leg convention \eqref{eq:sharp-convention}.
It leaves $\mathsf Y_3(r)=0$ unchanged.
Equivalently, the contraction formulas are
\begin{align}
 \langle\mathsf Y_2(r),\xi\otimes\eta\otimes\theta\rangle
  &=\langle\theta,[T\xi,T\eta]\rangle
    +\langle\xi,[T\eta,T\theta]\rangle
    +\langle\eta,[T\theta,T\xi]\rangle,
 \label{eq:Y2-contraction}\\
 \langle\mathsf Y_3(r),\xi\otimes\eta\otimes\zeta\otimes\theta\rangle
  &=\langle\theta,\tri{T\xi}{T\eta}{T\zeta}\rangle
    +\langle\xi,\tri{T\theta}{T\zeta}{T\eta}\rangle\nonumber\\
  &\quad-\langle\eta,\tri{T\theta}{T\zeta}{T\xi}\rangle
    -\langle\zeta,\tri{T\xi}{T\eta}{T\theta}\rangle.
 \label{eq:Y3-contraction}
\end{align}

Define the correction tensors
$\mathsf M_2(s)\in\g^{\otimes3}$ and
$\mathsf M_3(r,s)\in\g^{\otimes4}$ by
\begin{align}
 \langle\mathsf M_2(s),\xi\otimes\eta\otimes\theta\rangle
  &=\langle\theta,[S\xi,S\eta]\rangle,
 \label{eq:M2}\\
 \langle\mathsf M_3(r,s),\xi\otimes\eta\otimes\zeta\otimes\theta\rangle
  &=\big\langle\theta,
       \tri{T\xi}{S\eta}{S\zeta}
      +\tri{S\xi}{T\eta}{S\zeta}
      +\tri{S\xi}{S\eta}{T\zeta}
      -T c_S(\xi,\eta,\zeta)\big\rangle.
 \label{eq:M3}
\end{align}
These formulas fix the tensors and their normalizations.
The tensor $\mathsf M_3$ is linear in $r$ and quadratic in $s$.

\begin{definition}\label{def:extended-ybe}
Fix an admissible symmetric tensor $s$ and a mass $\kappa\in\K$.
The \emph{balanced extended Lie--Yamaguti Yang--Baxter equation}
for $r\in\wedge^2\g$ is
\begin{equation}\label{eq:extended-ybe}
 \boxed{\quad\mathsf Y_2(r)=\kappa\mathsf M_2(s),\qquad
        \mathsf Y_3(r)=\kappa\mathsf M_3(r,s).\quad}
\end{equation}
\end{definition}

\begin{theorem}\label{thm:ybe-operator}
A skew tensor $r$ solves \eqref{eq:extended-ybe} exactly when
$T=r^\sharp$ is an extended $\mathcal O$-operator of mass $\kappa$
with extension $S=s^\sharp$ for the coadjoint representation
\eqref{eq:coadjoint}. Equivalently, $\Gr(r^\sharp)$ is a
Lie--Yamaguti subalgebra of $E_{\kappa,S}=\g\oplus\g^*$.
The tensor $r$ may be degenerate.
\end{theorem}

\begin{proof}
Write $r=\sum_i x_i\otimes y_i$. Since $r$ is skew,
for every $\xi\in\g^*$,
\[
 T\xi=\sum_i\langle\xi,x_i\rangle y_i,\qquad
 \sum_i\langle\xi,y_i\rangle x_i=-T\xi.
\]
These give \eqref{eq:Y2-contraction}. The four ternary leg
products give, in order,
\begin{align*}
 \langle\tri{r_{31}}{r_{32}}{r_{43}},
                  \xi\otimes\eta\otimes\zeta\otimes\theta\rangle
   &=\langle\zeta,\tri{T\xi}{T\eta}{T\theta}\rangle,\\
 \langle\tri{r_{13}}{r_{41}}{r_{12}},
                  \xi\otimes\eta\otimes\zeta\otimes\theta\rangle
   &=-\langle\xi,\tri{T\theta}{T\zeta}{T\eta}\rangle,\\
 \langle\tri{r_{42}}{r_{23}}{r_{21}},
                  \xi\otimes\eta\otimes\zeta\otimes\theta\rangle
   &=\langle\eta,\tri{T\theta}{T\zeta}{T\xi}\rangle,\\
 \langle\tri{r_{41}}{r_{42}}{r_{43}},
                  \xi\otimes\eta\otimes\zeta\otimes\theta\rangle
   &=-\langle\theta,\tri{T\xi}{T\eta}{T\zeta}\rangle.
\end{align*}
The negative of their sum is \eqref{eq:Y3-contraction}.

Compare these with the operator equations.
Skew-symmetry of $T$ and the coadjoint formula give
\begin{align*}
 &\langle\theta,T(\rho^\vee(T\xi)\eta
                           -\rho^\vee(T\eta)\xi)\rangle\\
 &=-\langle\rho^\vee(T\xi)\eta-\rho^\vee(T\eta)\xi,
                                                    T\theta\rangle\\
 &=\langle\eta,[T\xi,T\theta]\rangle
                         -\langle\xi,[T\eta,T\theta]\rangle.
\end{align*}
Hence
\[
 \langle\mathsf Y_2(r),\xi\otimes\eta\otimes\theta\rangle
   =\langle\theta,[T\xi,T\eta]
       -T(\rho^\vee(T\xi)\eta-\rho^\vee(T\eta)\xi)\rangle.
\]
For the ternary terms, the same calculation gives
\begin{align*}
 \langle\theta,T\mu^\vee(T\eta,T\zeta)\xi\rangle
   &=-\langle\xi,\tri{T\theta}{T\zeta}{T\eta}\rangle,\\
 \langle\theta,T\mu^\vee(T\xi,T\zeta)\eta\rangle
   &=-\langle\eta,\tri{T\theta}{T\zeta}{T\xi}\rangle,\\
 \langle\theta,TD^\vee(T\xi,T\eta)\zeta\rangle
   &=\langle\zeta,\tri{T\xi}{T\eta}{T\theta}\rangle.
\end{align*}
By \eqref{eq:Y3-contraction}, the result is the pairing of
$\theta$ with
\[
 \tri{T\xi}{T\eta}{T\zeta}
 -T\bigl(\mu^\vee(T\eta,T\zeta)\xi
         -\mu^\vee(T\xi,T\zeta)\eta
         +D^\vee(T\xi,T\eta)\zeta\bigr).
\]
The tensors in \eqref{eq:M2}--\eqref{eq:M3} pair with the right
sides of the extended $\mathcal O$-operator equations.
Thus \eqref{eq:extended-ybe} is equivalent to
\eqref{eq:operator2}--\eqref{eq:operator3} for $T$.
Pairing with all covectors suffices; $T^{-1}$ is not needed.
The graph result follows from Theorem~\ref{thm:graph}.
\end{proof}

If $\kappa=0$ or $s=0$, this is the classical skew
Lie--Yamaguti Yang--Baxter equation. If $a^2=-\kappa$,
Proposition~\ref{prop:roots} gives the homomorphisms
\begin{equation}\label{eq:ybe-shifts}
 r^\sharp\pm a s^\sharp:(\g^*)_T\longrightarrow\g.
\end{equation}
This requires the three mixed ternary terms and $-Tc_S$
in \eqref{eq:M3}. A binary correction alone need not suffice.

We next treat quadratic Lie--Yamaguti algebras.

A quadratic Lie--Yamaguti algebra has a nondegenerate
symmetric form $\mathcal B$ satisfying
\begin{equation}\label{eq:quadratic-invariance}
 \mathcal B([x,y],z)=-\mathcal B(y,[x,z]),\qquad
 \mathcal B(\tri{x}{y}{z},w)
                 =\mathcal B(x,\tri{w}{z}{y}).
\end{equation}
Set $\mathcal B^\flat x=\mathcal B(x,\cdot)$.
Let $s_{\mathcal B}$ be the symmetric inverse tensor, with
$s_{\mathcal B}^\sharp=(\mathcal B^\flat)^{-1}$.

\begin{proposition}\label{prop:quadratic-ybe}
The tensor $s_{\mathcal B}$ is admissible.
For $r\in\wedge^2\g$, set $R=r^\sharp\mathcal B^\flat$.
Then \eqref{eq:extended-ybe} with $s=s_{\mathcal B}$ is
equivalent to $\mathcal B(Rx,y)=-\mathcal B(x,Ry)$ and
\begin{align}
 &[Rx,Ry]-R([Rx,y]+[x,Ry])=\kappa[x,y],
 \label{eq:quadratic-mrb2}\\
 &\tri{Rx}{Ry}{Rz}
 -R\bigl(\tri{x}{Ry}{Rz}+\tri{Rx}{y}{Rz}
                         +\tri{Rx}{Ry}{z}\bigr)\nonumber\\
 &\qquad=\kappa\bigl(\tri{Rx}{y}{z}+\tri{x}{Ry}{z}
                  +\tri{x}{y}{Rz}-R\tri{x}{y}{z}\bigr).
 \label{eq:quadratic-mrb3}
\end{align}
\end{proposition}

\begin{proof}
For all $x,y,z,w\in\g$, invariance gives
\begin{align*}
 \langle\mathcal B^\flat\rho_{\rm ad}(x)z,w\rangle
  &=\mathcal B([x,z],w)
   =-\mathcal B(z,[x,w])
   =\langle\rho^\vee(x)\mathcal B^\flat z,w\rangle,\\
 \langle\mathcal B^\flat\mu_{\rm ad}(x,y)z,w\rangle
  &=\mathcal B(\tri{z}{x}{y},w)
   =\mathcal B(z,\tri{w}{y}{x})
   =\langle\mu^\vee(x,y)\mathcal B^\flat z,w\rangle.
\end{align*}
Thus $\mathcal B^\flat$ identifies the adjoint and coadjoint
representations. It sends $S=(\mathcal B^\flat)^{-1}$ to
$\id_\g$, which is balanced equivariant for the adjoint
representation. Hence $s_{\mathcal B}$ is admissible.

Also,
\[
 \mathcal B(Rx,y)
   =\langle r,\mathcal B^\flat x\otimes\mathcal B^\flat y\rangle
   =-\langle r,\mathcal B^\flat y\otimes\mathcal B^\flat x\rangle
   =-\mathcal B(x,Ry).
\]
In Theorem~\ref{thm:ybe-operator}, substitute
$\xi=\mathcal B^\flat x$, $\eta=\mathcal B^\flat y$ and
$\zeta=\mathcal B^\flat z$. Then $T\xi=Rx$ and $S\xi=x$.
The binary input is $\mathcal B^\flat([Rx,y]+[x,Ry])$.
The ternary input is
\[
 \mathcal B^\flat\bigl(\tri{x}{Ry}{Rz}
            +\tri{Rx}{y}{Rz}+\tri{Rx}{Ry}{z}\bigr).
\]
Also, $c_S(\xi,\eta,\zeta)=\mathcal B^\flat\tri{x}{y}{z}$.
This gives \eqref{eq:quadratic-mrb2}--\eqref{eq:quadratic-mrb3}.
The substitutions are invertible, so the converse holds.
\end{proof}

\begin{theorem}\label{thm:lagrangian-graphs}
For $S=(\mathcal B^\flat)^{-1}$, the double
$E_{\kappa,S}$ is quadratic for the canonical pairing
\begin{equation}\label{eq:double-pairing}
 \mathcal Q((x,\xi),(y,\eta))
                   =\langle\xi,y\rangle+\langle\eta,x\rangle.
\end{equation}
Its maximal isotropic subalgebras complementary to
$\g\oplus0$ are exactly the graphs of solutions of
\eqref{eq:extended-ybe}.
\end{theorem}

\begin{proof}
Identify $\xi=\mathcal B^\flat u$. The double becomes
$\g\otimes A_\kappa$, where
$A_\kappa=\K[\varepsilon]/(\varepsilon^2+\kappa)$.
The form \eqref{eq:double-pairing} is
$\mathcal B\otimes q$, with
$q(a,b)=\lambda(ab)$, $\lambda(1)=0$,
$\lambda(\varepsilon)=1$. Its matrix on
$(1,\varepsilon)$ is $\left(\begin{smallmatrix}0&1\\1&0\end{smallmatrix}\right)$.
It is nondegenerate for every $\kappa$. The algebra $A_\kappa$
is associative and commutative. Together with
\eqref{eq:quadratic-invariance}, this gives invariance for
both products of the double.

Every complement to $\g\oplus0$ is the graph of a unique
$T:\g^*\to\g$. On the graph,
$\mathcal Q(j_T\xi,j_T\eta)=\langle\xi,T\eta\rangle+
\langle\eta,T\xi\rangle$. This vanishes exactly when
$T=r^\sharp$ for a skew tensor. The dimension count then gives
maximal isotropy. Closure under both products follows from
Theorem~\ref{thm:ybe-operator}.
\end{proof}

\begin{remark}\label{rem:manin-limitation}
Two complementary maximal isotropic subalgebras need not
satisfy the Lie--Yamaguti Manin-triple conditions of
\cite{ZhaoQiaoYB2023}. For $W=\g\oplus0$ and $Z=\Gr(T)$,
these also require
\[
 \tri{W}{W}{Z},\ \tri{W}{Z}{W}\subseteq Z,
 \qquad
 \tri{Z}{Z}{W},\ \tri{Z}{W}{Z}\subseteq W.
\]
Check these mixed ternary inclusions separately.
Equation~\eqref{eq:extended-ybe} alone need not give a
local-cocycle or double-construction Lie--Yamaguti bialgebra.
\end{remark}

We now give the Maurer--Cartan equation and skew deformations.

By Theorem~\ref{thm:ybe-operator} and
Section~\ref{sec:cohomology}, solutions correspond to skew
Maurer--Cartan elements $T:\g^*\to\g$.
For deformations of $r$, the low-degree complex must retain
skew-symmetry. An unrestricted operator correction need not
be skew.

For a solution $T=r^\sharp$, let $A^{\mathsf t}$ be
the ordinary transpose on $\g^*$. Define
\begin{align}
 K_T^0&=\{A\in\Der(\g):AS+SA^{\mathsf t}=0\},\nonumber\\
 K_T^1&=\{f:\g^*\to\g:
             \langle f\xi,\eta\rangle+
             \langle\xi,f\eta\rangle=0\},
 \label{eq:skew-complex-spaces}\\
 d_0A&=AT+TA^{\mathsf t}.\nonumber
\end{align}
Use the normal representation of $T$ and the cyclic map
\eqref{eq:delta-starI}--\eqref{eq:delta-starII} to form
\begin{equation}\label{eq:skew-complex}
 K_T^0\xrightarrow{d_0}K_T^1
 \xrightarrow{\delta_T}C_T^2
 \xrightarrow{(\delta_T^*,\delta_T)}W_T\oplus C_T^3
 \xrightarrow{(0,\delta_T)}C_T^4\longrightarrow\cdots.
\end{equation}
Denote its first and second cohomology groups by
$H^1_{r,s}$ and $H^2_{r,s}$.

\begin{proposition}\label{prop:skew-deformations}
Fix $(\g,s,\kappa)$. Infinitesimal skew solutions $r+t\dot r$
of \eqref{eq:extended-ybe}, modulo
$r_t\mapsto(\Phi_t\otimes\Phi_t)r_t$ with
$\Phi_t\in\Aut_{\rm LY}(\g)$, $(\Phi_t\otimes\Phi_t)s=s$
and $\Phi_0=\id$, are classified by $H^1_{r,s}$.
For a skew deformation of order $N$, \eqref{eq:obstruction}
gives a class in $H^2_{r,s}$. It extends to order $N+1$
through skew tensors exactly when this class vanishes.
\end{proposition}

\begin{proof}
By \eqref{eq:sharp-convention}, the tensor action is
$T\mapsto\Phi T\Phi^{\mathsf t}$. Its infinitesimal is
$AT+TA^{\mathsf t}$. The representation symmetry is
$(A,-A^{\mathsf t})$. Preserving $s$ requires
$AS+SA^{\mathsf t}=0$. Thus \eqref{eq:skew-complex} restricts
degrees zero and one of the augmented normal complex.
The first-order equation is $\delta_T\dot r^\sharp=0$.
By Theorem~\ref{thm:obstruction}, the obstruction is a cocycle
for $\delta_T$ and $\delta_T^*$. It is a boundary in
\eqref{eq:skew-complex} exactly when $\delta_Tf=-\Ob_{N+1}$
has a skew solution $f$. This is the required tensor correction.
\end{proof}

\begin{example}\label{ex:compact-ybe}
Let $\g=\mathfrak{so}(3)$ have an orthonormal basis
$e_1,e_2,e_3$ with $[e_1,e_2]=e_3$,
$[e_2,e_3]=e_1$, $[e_3,e_1]=e_2$, and set
$\tri{x}{y}{z}=[[x,y],z]$.
Take the standard invariant form $\mathcal B$,
$s=\sum_i e_i\otimes e_i$, $\kappa=1$, and
\[
 r=e_1\otimes e_2-e_2\otimes e_1,
 \qquad Re_1=e_2,\quad Re_2=-e_1,\quad Re_3=0.
\]
For each basis pair,
$[Rx,Ry]-R([Rx,y]+[x,Ry])=[x,y]$.
Proposition~\ref{prop:Lie} gives the ternary equation.
Thus $r$ is a nonzero degenerate solution of
\eqref{eq:extended-ybe}. The induced products are
\[
 [e_3,e_1]_R=-e_1,\qquad [e_3,e_2]_R=-e_2,
 \qquad [e_1,e_2]_R=0,\qquad
 \tri{x}{y}{z}_R=[[x,y]_R,z]_R.
\]
Its graph is maximal isotropic in $E_{1,S}$ for
\eqref{eq:double-pairing}.
\end{example}

\section{Maurer--Cartan characterization and cohomology}
\label{sec:cohomology}

We give the Maurer--Cartan equation and cohomology of extended
$\mathcal O$-operators. The insertion bracket gives a curved
$L_\infty[1]$-algebra. Its twisted differential is the paired
Yamaguti differential for the normal representation of a graph.
We include both cyclic maps in the lowest paired degree.
Degree zero comes from automorphisms of the fixed data.
The algebra, representation, balanced map $S$ and mass $\kappa$
remain fixed.

For a vector space $U$, put
\begin{align}
 \cC^0(U,U)&=\End(U),\nonumber\\
 \cC^p(U,U)
 &=\Hom((\wedge^2U)^{\otimes p},U)
   \oplus\Hom((\wedge^2U)^{\otimes p}\otimes U,U),\qquad p\geq1.
 \label{eq:graded-space}
\end{align}
We use the insertion bracket of \cite{ZhaoQiao2023} with the
following convention. Write $X_i=x_i\wedge y_i$,
$P=(P_{\rm I},P_{\rm II})\in\cC^p$ and
$Q=(Q_{\rm I},Q_{\rm II})\in\cC^q$, with $p,q\geq1$.
Then
\begin{align}
 &(P\circ Q)_{\rm I}(X_1,\ldots,X_{p+q})\nonumber\\
 &=(-1)^{pq}\!
 \sum_{\substack{\sigma\in\Sh(p,q)\\\sigma(p+q)=p+q}}
 \sgn(\sigma)P_{\rm II}
 \bigl(X_{\sigma(1)},\ldots,X_{\sigma(p)},
       Q_{\rm I}(X_{\sigma(p+1)},\ldots,X_{\sigma(p+q)})\bigr)
 \nonumber\\
 &\quad+\sum_{k=1}^p(-1)^{(k-1)q}
       \sum_{\sigma\in\Sh(k-1,q)}\sgn(\sigma)\,
       P_{\rm I}\Bigl(X_{\sigma(1)},\ldots,X_{\sigma(k-1)},\nonumber\\
 &\hspace{5em}
 Q_{\rm II}(X_{\sigma(k)},\ldots,X_{\sigma(k+q-1)},x_{k+q})
          \wedge y_{k+q}\nonumber\\
 &\hspace{5em}
 +x_{k+q}\wedge
 Q_{\rm II}(X_{\sigma(k)},\ldots,X_{\sigma(k+q-1)},y_{k+q}),
 X_{k+q+1},\ldots,X_{p+q}\Bigr),\label{eq:insertI}\\
 &(P\circ Q)_{\rm II}(X_1,\ldots,X_{p+q},z)\nonumber\\
 &=(-1)^{pq}\sum_{\sigma\in\Sh(p,q)}\sgn(\sigma)
 P_{\rm II}\bigl(X_{\sigma(1)},\ldots,X_{\sigma(p)},
 Q_{\rm II}(X_{\sigma(p+1)},\ldots,X_{\sigma(p+q)},z)\bigr)
 \nonumber\\
 &\quad+\sum_{k=1}^p(-1)^{(k-1)q}
       \sum_{\sigma\in\Sh(k-1,q)}\sgn(\sigma)\,
       P_{\rm II}\Bigl(X_{\sigma(1)},\ldots,X_{\sigma(k-1)},\nonumber\\
 &\hspace{5em}
 Q_{\rm II}(X_{\sigma(k)},\ldots,X_{\sigma(k+q-1)},x_{k+q})
          \wedge y_{k+q}\nonumber\\
 &\hspace{5em}
 +x_{k+q}\wedge
 Q_{\rm II}(X_{\sigma(k)},\ldots,X_{\sigma(k+q-1)},y_{k+q}),
 X_{k+q+1},\ldots,X_{p+q},z\Bigr).\label{eq:insertII}
\end{align}
A shuffle preserves the order within each block.
In each inner sum, it permutes only the first $k+q-1$ positions.
For $f\in\End(U)$, $P\circ f$ inserts $f$ into each input,
including the last input of $P_{\rm II}$.
The expression $f\circ P$ applies $f$ to the output.
For endomorphisms, $\circ$ is composition. Set
\begin{equation}\label{eq:graded-bracket}
 [P,Q]_{\rm LY}=P\circ Q-(-1)^{pq}Q\circ P.
\end{equation}
This is the graded Lie bracket of \cite{ZhaoQiao2023}.
For a Lie--Yamaguti structure $\Delta=(\pi,\omega)$,
$[\Delta,\Delta]_{\rm LY}=0$. Its components are twice the
differences in \eqref{eq:LY3} and \eqref{eq:LY4}.

\begin{remark}\label{rem:MC-caution}
For an arbitrary pair $(\pi,\omega)$,
$[\Delta,\Delta]_{\rm LY}=0$ need not imply
\eqref{eq:LY1} or \eqref{eq:LY2}.
Our double satisfies all four identities by
Theorem~\ref{thm:double}. The Maurer--Cartan equation below
describes graph closure in this fixed double.
\end{remark}

Let $E=E_{\kappa,S}$ and let $\Delta$ be its products.
Define $\mathfrak a\subset\cC^\bullet(E,E)$ by lifting maps
on $V$ with values in $\g$. Use $E\to V$ on each input
and $\g\to E$ on the output. Thus
\begin{align}
 \mathfrak a^0&=\Hom(V,\g),\nonumber\\
 \mathfrak a^p
 &=\Hom((\wedge^2V)^{\otimes p},\g)
  \oplus\Hom((\wedge^2V)^{\otimes p}\otimes V,\g),\quad p\geq1.
 \label{eq:a}
\end{align}
Let $P:\cC^\bullet(E,E)\to\mathfrak a$ restrict inputs to $V$
and project the output to $\g$.

\begin{theorem}\label{thm:curved}
The maps
\begin{equation}\label{eq:derived}
 \ell_0=P\Delta,\qquad
 \ell_n(f_1,\ldots,f_n)
 =P[\cdots[[\Delta,\widehat f_1]_{\rm LY},
               \widehat f_2]_{\rm LY},\ldots,\widehat f_n]_{\rm LY}
 \quad(n\geq1)
\end{equation}
make $\mathfrak a$ a curved $L_\infty[1]$-algebra.
Here $\widehat f$ is the lift above. Each $\ell_n$ has degree
one and is graded symmetric. The brackets vanish for $n\geq4$.
The curvature is
\begin{equation}\label{eq:curvature}
 \ell_0=\bigl((u,v)\mapsto-\kappa[Su,Sv],\,0\bigr).
\end{equation}
A degree-zero map $T:V\to\g$ is an extended $\mathcal O$-operator
exactly when
\begin{equation}\label{eq:MC}
 \ell_0+\ell_1(T)+\frac12\ell_2(T,T)
                         +\frac16\ell_3(T,T,T)=0.
\end{equation}
\end{theorem}

\begin{proof}
The lifted subspace is abelian. Its outputs lie in $\g$,
so inserting them into another lifted cochain gives zero.
The kernel of $P$ is a graded Lie subalgebra: insertion
preserves cochains that send inputs in $V$ to $V$.
The double identities give $[\Delta,\Delta]_{\rm LY}=0$.
Thus the curved higher derived bracket theorem
\cite{Voronov2005} applies. It does not require $P\Delta=0$.
Evaluating $\Delta$ on $V$ gives \eqref{eq:curvature}.

Each product has at most three inputs. Two lifted cochains
cannot compose, so every nonzero insertion meets an input
or the output of $\Delta$. A fourth insertion needs three
$\g$-inputs in the ternary product, followed by insertion
of its output into a lifted cochain. That output lies in
$\g$, since $\g\oplus0$ is a subalgebra. The term is therefore
zero. The binary product allows at most two insertions.

Extend $T$ by $\widehat T(x,u)=(Tu,0)$.
Since $\widehat T^2=0$, exponentiating the right adjoint
action in \eqref{eq:derived} gives
\begin{align}
 \Delta^T_{\rm I}(X,Y)
 &=(\id-\widehat T)
       [(\id+\widehat T)X,(\id+\widehat T)Y]_E,\nonumber\\
 \Delta^T_{\rm II}(X,Y,Z)
 &=(\id-\widehat T)
       \tri{(\id+\widehat T)X}{(\id+\widehat T)Y}
                                   {(\id+\widehat T)Z}_E.
 \label{eq:transport}
\end{align}
Thus the left side of \eqref{eq:MC} is $P\Delta^T$.
Its binary component on $u,v$ is
\[
 [Tu,Tv]-\kappa[Su,Sv]
                  -T(\rho(Tu)v-\rho(Tv)u),
\]
and its ternary component is
\begin{align*}
 &\tri{Tu}{Tv}{Tw}
   -\kappa(\tri{Tu}{Sv}{Sw}+\tri{Su}{Tv}{Sw}
                                      +\tri{Su}{Sv}{Tw})\\
 &\qquad-T\bigl(\mu(Tv,Tw)u-\mu(Tu,Tw)v+D(Tu,Tv)w
                                      -\kappa c(u,v,w)\bigr).
\end{align*}
These vanish exactly when the two operator equations hold.
\end{proof}

For degree-zero maps $f,g,h:V\to\g$, the brackets are
\begin{align}
 (\ell_1(f))_{\rm I}(u,v)&=0,\nonumber\\
 (\ell_1(f))_{\rm II}(u,v,w)
 &=\kappa f(c(u,v,w))
   -\kappa\bigl(\tri{fu}{Sv}{Sw}+\tri{Su}{fv}{Sw}
                                     +\tri{Su}{Sv}{fw}\bigr),
 \label{eq:l1}\\
 (\ell_2(f,g))_{\rm I}(u,v)
 &=[fu,gv]+[gu,fv]\nonumber\\
 &\quad-f(\rho(gu)v-\rho(gv)u)
           -g(\rho(fu)v-\rho(fv)u),\nonumber\\
 (\ell_2(f,g))_{\rm II}(u,v,w)&=0,\label{eq:l2}\\
 (\ell_3(f,g,h))_{\rm I}(u,v)&=0,\nonumber\\
 (\ell_3(f,g,h))_{\rm II}(u,v,w)
 &=\sum_{\sigma\in S_3}\Bigl(
       \tri{f_{\sigma(1)}u}{f_{\sigma(2)}v}{f_{\sigma(3)}w}
       \nonumber\\
 &\quad-f_{\sigma(1)}
       \bigl(\mu(f_{\sigma(2)}v,f_{\sigma(3)}w)u
              -\mu(f_{\sigma(2)}u,f_{\sigma(3)}w)v\nonumber\\
 &\hspace{9em}+D(f_{\sigma(2)}u,f_{\sigma(3)}v)w\bigr)\Bigr),
 \label{eq:l3}
\end{align}
where $(f_1,f_2,f_3)=(f,g,h)$. At nonzero mass, the constant
binary and linear ternary terms in \eqref{eq:MC} need not
vanish separately.

\begin{corollary}\label{cor:twist}
Fix an extended $\mathcal O$-operator $T$. The brackets
\begin{align}
 \ell_1^T(f)&=\ell_1(f)+\ell_2(T,f)+\tfrac12\ell_3(T,T,f),
 \nonumber\\
 \ell_2^T(f,g)&=\ell_2(f,g)+\ell_3(T,f,g),\qquad
 \ell_3^T(f,g,h)=\ell_3(f,g,h)\label{eq:twist}
\end{align}
give an uncurved $L_\infty[1]$-algebra.
The map $T+f$ is an extended $\mathcal O$-operator exactly when
\begin{equation}\label{eq:twistedMC}
 \ell_1^T(f)+\tfrac12\ell_2^T(f,f)
                            +\tfrac16\ell_3^T(f,f,f)=0.
\end{equation}
In particular, $(\ell_1^T)^2=0$.
\end{corollary}

\begin{proof}
The lifts of $T$ and $f$ lie in the abelian subspace
$\mathfrak a$. Their right adjoint actions commute.
Thus transport by $T$ followed by $f$ equals transport by
$T+f$. Expanding \eqref{eq:MC} at $T+f$ gives
\begin{align*}
 &\ell_0+\ell_1(T+f)+\tfrac12\ell_2(T+f,T+f)
                         +\tfrac16\ell_3(T+f,T+f,T+f)\\
 &=\ell_0+\ell_1(T)+\tfrac12\ell_2(T,T)
                         +\tfrac16\ell_3(T,T,T)\\
 &\quad+\ell_1(f)+\ell_2(T,f)+\tfrac12\ell_3(T,T,f)
       +\tfrac12\ell_2(f,f)+\tfrac12\ell_3(T,f,f)
       +\tfrac16\ell_3(f,f,f).
\end{align*}
The terms on the first line of the right side sum to zero,
since $T$ is an extended $\mathcal O$-operator.
Grouping the rest by the number of copies of $f$ gives
\eqref{eq:twist} and \eqref{eq:twistedMC}.

The transported structure $\Delta^T$ satisfies
$[\Delta^T,\Delta^T]_{\rm LY}=0$.
Graph closure gives $P\Delta^T=0$.
The derived bracket construction for $\Delta^T$ gives
\eqref{eq:twist} with zero curvature.
Its unary $L_\infty[1]$ identity gives $\ell_1^T\ell_1^T=0$.
\end{proof}

We next find the representation on the quotient by a graph.

\begin{theorem}\label{thm:normal}
For an extended $\mathcal O$-operator $T$, define maps on $\g$ by
\begin{align}
 \rho_T(u)x&=[Tu,x]+T\rho(x)u,\label{eq:normal-rho}\\
 \mu_T(u,v)x
 &=\tri{x}{Tu}{Tv}-\kappa\tri{x}{Su}{Sv}
             +T\bigl(\mu(x,Tv)u-D(x,Tu)v\bigr).
 \label{eq:normal-mu}
\end{align}
These form a representation of $V_T$.
Its associated skew map is
\begin{align}
 D_T(u,v)x
 &=\tri{Tu}{Tv}{x}-\kappa\tri{Su}{Sv}{x}
                 -T\bigl(\mu(Tv,x)u-\mu(Tu,x)v\bigr).
 \label{eq:normal-D}
\end{align}
It is the adjoint action of $\Gr(T)$ on
$E_{\kappa,S}/\Gr(T)$ under the identification
\begin{equation}\label{eq:quotient}
 q_T:E_{\kappa,S}\longrightarrow\g,\qquad q_T(x,u)=x-Tu.
\end{equation}
\end{theorem}

\begin{proof}
Restrict the adjoint representation of $E$ to $\Gr(T)$ acting
on $E$. Closure makes $\Gr(T)$ a submodule. Thus changing a
representative by an element of $\Gr(T)$ changes the output
by an element of $\Gr(T)$. The quotient $E/\Gr(T)$ inherits
all five representation identities from $E$.
For representatives $(x,0)$,
\begin{align*}
 q_T[j_T(u),(x,0)]_E
   &=[Tu,x]+T\rho(x)u,\\
 q_T\tri{(x,0)}{j_T(u)}{j_T(v)}_E
   &=\tri{x}{Tu}{Tv}-\kappa\tri{x}{Su}{Sv}
       +T\bigl(\mu(x,Tv)u-D(x,Tu)v\bigr),\\
 q_T\tri{j_T(u)}{j_T(v)}{(x,0)}_E
   &=\tri{Tu}{Tv}{x}-\kappa\tri{Su}{Sv}{x}
       -T\bigl(\mu(Tv,x)u-\mu(Tu,x)v\bigr).
\end{align*}
These give the three formulas. The first Lie--Yamaguti identity
on the quotient gives
$D_T(u,v)=\mu_T(v,u)-\mu_T(u,v)
+[\rho_T(u),\rho_T(v)]-\rho_T([u,v]_T)$.
The subalgebra $\Gr(T)$ need not be an ideal of $E$.
\end{proof}

First take the paired complex from the insertion bracket.
We add the low-degree cyclic maps below. For $n\geq1$, set
\begin{equation}\label{eq:cochains}
 C_T^1=\Hom(V,\g),\qquad
 C_T^{n+1}=
 \Hom((\wedge^2V)^{\otimes n},\g)
 \oplus\Hom((\wedge^2V)^{\otimes n}\otimes V,\g).
\end{equation}
For $f\in C_T^1$, define
\begin{align}
 (\delta_T f)_{\rm I}(u,v)
 &=\rho_T(u)f(v)-\rho_T(v)f(u)-f([u,v]_T),\label{eq:delta1I}\\
 (\delta_T f)_{\rm II}(u,v,w)
 &=D_T(u,v)f(w)+\mu_T(v,w)f(u)-\mu_T(u,w)f(v)
                                      -f(\Tri{u}{v}{w}).
 \label{eq:delta1II}
\end{align}
For $X_i=u_i\wedge v_i$ write
\[
 X_i\circ_T X_j=\Tri{u_i}{v_i}{u_j}\wedge v_j
                         +u_j\wedge\Tri{u_i}{v_i}{v_j}.
\]
A hat $\widehat X_k$ means that the input is omitted.
Write $D_T(X_k)=D_T(u_k,v_k)$. For $(f,g)\in C_T^{n+1}$
with $n\geq1$, define the Yamaguti differential by
\begin{align}
 &(\delta_T(f,g))_{\rm I}(X_1,\ldots,X_{n+1})\nonumber\\
 &=(-1)^n\bigl(
       \rho_T(u_{n+1})g(X_1,\ldots,X_n,v_{n+1})
      -\rho_T(v_{n+1})g(X_1,\ldots,X_n,u_{n+1})\nonumber\\
 &\hspace{5em}-g(X_1,\ldots,X_n,[u_{n+1},v_{n+1}]_T)\bigr)
 \nonumber\\
 &\quad+\sum_{k=1}^{n}(-1)^{k+1}D_T(X_k)
       f(X_1,\ldots,\widehat X_k,\ldots,X_{n+1})\nonumber\\
 &\quad+\sum_{1\leq k<j\leq n+1}(-1)^k
       f(X_1,\ldots,\widehat X_k,\ldots,
                        X_k\circ_T X_j,\ldots,X_{n+1}),
 \label{eq:delta-highI}\\
 &(\delta_T(f,g))_{\rm II}(X_1,\ldots,X_{n+1},w)\nonumber\\
 &=(-1)^n\bigl(
    \mu_T(v_{n+1},w)g(X_1,\ldots,X_n,u_{n+1})
    -\mu_T(u_{n+1},w)g(X_1,\ldots,X_n,v_{n+1})\bigr)
 \nonumber\\
 &\quad+\sum_{k=1}^{n+1}(-1)^{k+1}D_T(X_k)
       g(X_1,\ldots,\widehat X_k,\ldots,X_{n+1},w)\nonumber\\
 &\quad+\sum_{1\leq k<j\leq n+1}(-1)^k
       g(X_1,\ldots,\widehat X_k,\ldots,
                        X_k\circ_T X_j,\ldots,X_{n+1},w)
 \nonumber\\
 &\quad+\sum_{k=1}^{n+1}(-1)^k
       g(X_1,\ldots,\widehat X_k,\ldots,X_{n+1},
                                             \Tri{u_k}{v_k}{w}).
 \label{eq:delta-highII}
\end{align}
In the third sums, $X_k\circ_T X_j$ replaces $X_j$
after $X_k$ is omitted.

\begin{theorem}\label{thm:differential}
For $F\in C_T^{n+1}=\mathfrak a^n$,
\begin{equation}\label{eq:sign}
 \ell_1^T(F)=(-1)^n\delta_TF\qquad(n\geq0).
\end{equation}
Hence $\delta_T^2=0$. In degree one, this is the derivative
of the two operator equations with the data fixed.
\end{theorem}

\begin{proof}
Transport by \eqref{eq:transport} makes $V$ the subalgebra
$V_T$. The quotient action is
\eqref{eq:normal-rho}--\eqref{eq:normal-D}.
The twisted unary bracket is $P[\Delta^T,\widehat F]_{\rm LY}$.
In \eqref{eq:insertI}--\eqref{eq:insertII}, the last pair gives
the first lines of \eqref{eq:delta-highI}--\eqref{eq:delta-highII}.
The output action gives the sums with $D_T$.
Insertion into a wedge pair gives the sums with $X_k\circ_T X_j$.
Insertion into the last input gives the final sum in
\eqref{eq:delta-highII}. The signs are $(-1)^n$ times the
Yamaguti signs. This proves \eqref{eq:sign} for $n\geq1$.

For $f\in C_T^1$, the binary component is
\begin{align*}
 [fu,Tv]+[Tu,fv]
 -f(\rho(Tu)v-\rho(Tv)u)-T(\rho(fu)v-\rho(fv)u).
\end{align*}
The ternary component is
\begin{align}
 &\tri{fu}{Tv}{Tw}+\tri{Tu}{fv}{Tw}+\tri{Tu}{Tv}{fw}
       -f(\Tri{u}{v}{w})\nonumber\\
 &\quad-\kappa\bigl(\tri{fu}{Sv}{Sw}
                     +\tri{Su}{fv}{Sw}+\tri{Su}{Sv}{fw}\bigr)
 \nonumber\\
 &\quad-T\bigl(\mu(fv,Tw)u+\mu(Tv,fw)u
                -\mu(fu,Tw)v-\mu(Tu,fw)v\nonumber\\
 &\hspace{9em}+D(fu,Tv)w+D(Tu,fv)w\bigr).
 \label{eq:linearization}
\end{align}
Substituting the normal representation in
\eqref{eq:delta1I}--\eqref{eq:delta1II} gives these formulas.
They are the coefficients of $t$ in the equations for $T+tf$.
This proves the graded degree-zero case.
Now \eqref{eq:sign} and $(\ell_1^T)^2=0$ give $\delta_T^2=0$.
\end{proof}

Let $\mathfrak s$ be the Lie algebra of pairs
$(A,B)\in\Der(\g)\oplus\End(V)$ satisfying
\begin{align}
 AS&=SB,\label{eq:sym1}\\
 B\rho(x)-\rho(x)B&=\rho(Ax),\label{eq:sym2}\\
 B\mu(x,y)-\mu(x,y)B&=\mu(Ax,y)+\mu(x,Ay).\label{eq:sym3}
\end{align}
Here $A$ is a derivation of both operations of $\g$.
Set
\begin{equation}\label{eq:degreezero}
 C_T^0=\mathfrak s,\qquad
 \partial_T(A,B)=AT-TB,\qquad
 \partial_T|_{C_T^{n}}=\delta_T\quad(n\geq1).
\end{equation}

\begin{proposition}\label{prop:cohom}
The maps in \eqref{eq:degreezero} form a cochain complex.
Define
\begin{equation}\label{eq:H}
 H_T^n=\frac{\ker(\partial_T:C_T^n\to C_T^{n+1})}
                 {\im(\partial_T:C_T^{n-1}\to C_T^n)}
 \quad(n\geq1),\qquad H_T^0=\ker\partial_T|_{C_T^0}.
\end{equation}
For $n\geq3$, these are the paired Yamaguti groups of $V_T$
with normal coefficients. The group $H_T^2$ uses the paired
insertion complex. The full $(2,3)$-group also requires
\eqref{eq:delta-starI}--\eqref{eq:delta-starII}.
Degree one takes the quotient by the chosen symmetries of
the fixed data.
\end{proposition}

\begin{proof}
By \eqref{eq:sym1}--\eqref{eq:sym3}, $(e^{tA},e^{tB})$
preserves $S$ and the representation. Thus $e^{tA}Te^{-tB}$
is an extended $\mathcal O$-operator if $T$ is one.
Its coefficient of $t$ is $AT-TB$.
Theorem~\ref{thm:differential} gives $\delta_T(AT-TB)=0$.
Together with $\delta_T^2=0$, this proves the claim.
\end{proof}

Full Lie--Yamaguti cohomology also needs the cyclic conditions
\eqref{eq:LY1} and \eqref{eq:LY2} in the lowest paired degree.
This matters for deformations \cite{Goswami2025}.
We give the extra maps with the normal coefficients of
Theorem~\ref{thm:normal}.

Write $B(u,v)=[u,v]_T$ and $C(u,v,w)=\Tri{u}{v}{w}$. Set
\[
 W_T=\Hom(\wedge^3V,\g)
       \oplus\Hom(\wedge^3V\otimes V,\g).
\]
The second summand is alternating in its first three inputs,
not in two consecutive pairs. It is an auxiliary cyclic cochain
space. For $(F,G)\in C_T^2$, define $\delta_T^*(F,G)\in W_T$ by
\begin{align}
 (\delta_T^*(F,G))_{\rm I}(u,v,w)
 &=\cyc_{u,v,w}\bigl(-\rho_T(u)F(v,w)
          +F(B(u,v),w)+G(u,v,w)\bigr),
 \label{eq:delta-starI}\\
 (\delta_T^*(F,G))_{\rm II}(u,v,w,z)
 &=\cyc_{u,v,w}\bigl(\mu_T(u,z)F(v,w)
                         +G(B(u,v),w,z)\bigr).
 \label{eq:delta-starII}
\end{align}
Both components use the pair $(F,G)$.
Their common kernel need not split into two cocycle spaces.

Equations \eqref{eq:delta-highI}--\eqref{eq:delta-highII}
give the other two low-degree equations:
\begin{align}
 (\delta_T(F,G))_{\rm I}(u,v,w,z)
 &=-\rho_T(w)G(u,v,z)+\rho_T(z)G(u,v,w)
                  +G(u,v,B(w,z))\nonumber\\
 &\quad+D_T(u,v)F(w,z)
       -F(C(u,v,w),z)-F(w,C(u,v,z)),
 \label{eq:delta-pairI}\\
 (\delta_T(F,G))_{\rm II}(u,v,w,z,a)
 &=-\mu_T(z,a)G(u,v,w)+\mu_T(w,a)G(u,v,z)\nonumber\\
 &\quad+D_T(u,v)G(w,z,a)-D_T(w,z)G(u,v,a)\nonumber\\
 &\quad-G(C(u,v,w),z,a)-G(w,C(u,v,z),a)\nonumber\\
 &\quad-G(w,z,C(u,v,a))+G(u,v,C(w,z,a)).
 \label{eq:delta-pairII}
\end{align}

\begin{proposition}\label{prop:cyclic-complex}
On $C_T^1$, $\delta_T^*\delta_T=0$.
Thus the augmented sequence below is a cochain complex:
\begin{equation}\label{eq:augmented-complex}
 \mathfrak s\xrightarrow{\partial_T}C_T^1
 \xrightarrow{\delta_T}C_T^2
 \xrightarrow{(\delta_T^*,\delta_T)}W_T\oplus C_T^3
 \xrightarrow{(0,\delta_T)}C_T^4
 \xrightarrow{\delta_T}C_T^5\longrightarrow\cdots.
\end{equation}
Denote its cohomology by $\widehat H_T^\bullet$, in degrees
$0,1,2,3,4,\ldots$. Then
\begin{align}
 \widehat H_T^0=H_T^0
   &=\{(A,B)\in\mathfrak s:AT=TB\},\label{eq:H0-explicit}\\
 \widehat H_T^1=H_T^1
   &=\frac{\{f\in\Hom(V,\g):\delta_Tf=0\}}
                 {\{AT-TB:(A,B)\in\mathfrak s\}},
 \label{eq:H1-explicit}\\
 \widehat H_T^2
   &=\frac{\{(F,G):\delta_T(F,G)=0,
                         \ \delta_T^*(F,G)=0\}}
                 {\{\delta_T f:f\in\Hom(V,\g)\}}.
 \label{eq:H2-corrected}
\end{align}
There is a natural injection $\widehat H_T^2\hookrightarrow H_T^2$.
The group $\widehat H_T^2$ is the usual $(2,3)$-cohomology
with normal coefficients and both cyclic conditions.
\end{proposition}

\begin{proof}
Let $f:V\to\g$. Take the semidirect product of $V_T$ with
its normal representation. Both products on the second summand
are zero. Set $j_t(u)=(u,tf(u))$, $F=(\delta_Tf)_{\rm I}$
and $G=(\delta_Tf)_{\rm II}$. Substitution gives
\begin{align*}
 [j_t(u),j_t(v)]&=j_t(B(u,v))+(0,tF(u,v)),\\
 \tri{j_t(u)}{j_t(v)}{j_t(w)}
    &=j_t(C(u,v,w))+(0,tG(u,v,w)).
\end{align*}
Apply the first cyclic identity to three such inputs.
Subtract $tf$ of the first component from the second. This gives
\[
 t\cyc_{u,v,w}
  \bigl(F(B(u,v),w)-\rho_T(w)F(u,v)+G(u,v,w)\bigr)=0.
\]
The coefficient of $t$ is $(\delta_T^*(F,G))_{\rm I}=0$.
Do the same for the second cyclic identity on
$j_t(u),j_t(v),j_t(w),j_t(z)$. This gives
\[
 t\cyc_{u,v,w}
  \bigl(G(B(u,v),w,z)+\mu_T(w,z)F(u,v)\bigr)=0.
\]
Cycling the summands gives $(\delta_T^*(F,G))_{\rm II}=0$.
Hence $\delta_T^*\delta_Tf=0$.

The composite from $C_T^1$ to $W_T\oplus C_T^3$ is zero:
its components are $\delta_T^*\delta_T$ and $\delta_T^2$.
The next composite is
$(0,\delta_T)(\delta_T^*,\delta_T)=\delta_T^2=0$.
Later composites vanish by Theorem~\ref{thm:differential}.
The degree-zero composite vanishes by Proposition~\ref{prop:cohom}.
Thus this is a complex.

The kernels and images in degrees zero, one and two give
the stated formulas. A corrected $2$-cocycle that is zero
in $H_T^2$ equals $\delta_Tf$ for some $f\in C_T^1$.
The same $f$ makes it a boundary in $\widehat H_T^2$.
Thus the map to $H_T^2$ is injective.
\end{proof}

Thus $H_T^0$ is the infinitesimal stabilizer of $T$ under
the chosen symmetries. The group $H_T^1$ describes first-order
normal motions of the graph. An augmented $2$-cocycle satisfies
all four equations \eqref{eq:delta-starI}--\eqref{eq:delta-pairII}.
A $2$-coboundary uses the same map $f$ in both
\eqref{eq:delta1I}--\eqref{eq:delta1II}.

\begin{example}\label{ex:missing-cyclic}
Take $V=\K^3$, $\g=\K a$ with zero products, the zero
representation, $S=0$ and $T=0$. Then $\delta_T$ is zero in
every positive degree. Set $F=0$. Let $G$ have only the nonzero
basis values $G(e_1,e_2,e_3)=a$ and $G(e_2,e_1,e_3)=-a$.
The pair $(F,G)$ is a cocycle for the insertion complex, but
\[
 (\delta_T^*(F,G))_{\rm I}(e_1,e_2,e_3)=a\ne0.
\]
Thus it is not a full $(2,3)$-cocycle.
The groups $H_T^2$ and $\widehat H_T^2$ can differ.
\end{example}

\begin{lemma}\label{lem:cyclic-defect}
For any linear $U:V\to\g$, let $B_U,C_U$ be
\eqref{eq:desc2}--\eqref{eq:desc3}. The map $U$ need not solve
the operator equations. Define its graph defects
$\mathcal F_U,\mathcal G_U$ by
\begin{align*}
 [j_Uu,j_Uv]_E
   &=j_U B_U(u,v)+(\mathcal F_U(u,v),0),\\
 \tri{j_Uu}{j_Uv}{j_Uw}_E
   &=j_U C_U(u,v,w)+(\mathcal G_U(u,v,w),0).
\end{align*}
Use \eqref{eq:normal-rho}--\eqref{eq:normal-mu} to define
$\rho_U,\mu_U$, and \eqref{eq:delta-starI}--\eqref{eq:delta-starII}
to define $\delta_U^*$. Even if these maps are not a
representation,
\begin{equation}\label{eq:cyclic-defect}
 \delta_U^*(\mathcal F_U,\mathcal G_U)=0.
\end{equation}
\end{lemma}

\begin{proof}
Apply $q_U(x,u)=x-Uu$ to
$\cyc([[j_Uu,j_Uv]_E,j_Uw]_E+
\tri{j_Uu}{j_Uv}{j_Uw}_E)=0$. This gives
\[
 \cyc\bigl(\mathcal F_U(B_U(u,v),w)
             -\rho_U(w)\mathcal F_U(u,v)
             +\mathcal G_U(u,v,w)\bigr)=0,
\]
which is the first component. Applying $q_U$ to
$\cyc\tri{[j_Uu,j_Uv]_E}{j_Uw}{j_Uz}_E=0$
gives the second, since
$q_U\tri{(x,0)}{j_Uw}{j_Uz}_E=\mu_U(w,z)x$.
\end{proof}

The cyclic maps are compatible with this operator problem.
The double satisfies all four Lie--Yamaguti identities, and
its graph defects always satisfy \eqref{eq:cyclic-defect}.
Thus the Maurer--Cartan description still holds.
The extra maps give a more precise obstruction group;
see Theorem~\ref{thm:obstruction}.

\section{Deformations and obstructions}
\label{sec:deformations}

We study deformations of an extended $\mathcal O$-operator
with the algebra, representation, extension map and mass fixed.
We classify infinitesimals, prove rigidity and extension criteria,
and compare deformations of the induced algebra.

\begin{definition}\label{def:deformation}
A formal deformation of an extended $\mathcal O$-operator $T$ is a series
\begin{equation}\label{eq:formal}
 T_t=\sum_{i\geq0}t^iT_i,\qquad T_0=T,\quad T_i\in\Hom(V,\g),
\end{equation}
satisfying \eqref{eq:operator2}--\eqref{eq:operator3} at each
coefficient. A deformation of order $N$ satisfies them modulo
$t^{N+1}$. An infinitesimal deformation is a solution over
$\K[t]/(t^2)$.
\end{definition}

Let $\mathcal G$ consist of pairs
$(\Phi,\Psi)\in\Aut_{\rm LY}(\g)\times\operatorname{GL}(V)$
such that
\begin{align}
 \Phi S&=S\Psi,\nonumber\\
 \Psi\rho(x)&=\rho(\Phi x)\Psi,\qquad
 \Psi\mu(x,y)=\mu(\Phi x,\Phi y)\Psi.
 \label{eq:group}
\end{align}
Its Lie algebra is $\mathfrak s$ in
\eqref{eq:sym1}--\eqref{eq:sym3}. It acts by
$T\mapsto\Phi T\Psi^{-1}$.
Two formal deformations are equivalent if
\begin{equation}\label{eq:equivalence}
 T'_t=\Phi_tT_t\Psi_t^{-1},
 \qquad(\Phi_t,\Psi_t)\in\mathcal G(\K[[t]]),
 \qquad(\Phi_0,\Psi_0)=(\id,\id).
\end{equation}
Over the dual numbers, this defines infinitesimal equivalence.
Each pair gives a block diagonal automorphism of $E_{\kappa,S}$.
Thus the action preserves graph closure.

\begin{theorem}\label{thm:infinitesimals}
The infinitesimal deformations of $T$ are the $1$-cocycles
$\ker\delta_T$. Their equivalence classes form $H_T^1$.
If $T_t$ is formal and $T_m$ is its first nonzero coefficient
after $T_0$, then $\delta_TT_m=0$.
\end{theorem}

\begin{proof}
Substitute $T+tf$ in the operator equations. By
Theorem~\ref{thm:differential}, the coefficient of $t$ is
\eqref{eq:delta1I}--\eqref{eq:delta1II}.
Thus $f$ is infinitesimal exactly when it is a cocycle.
For $\Phi_t=\id+tA$ and $\Psi_t=\id+tB$ modulo $t^2$,
\[
 (\id+tA)(T+tf)(\id-tB)
       =T+t(f+AT-TB)\pmod{t^2}.
\]
The linear terms of \eqref{eq:group} require $(A,B)\in\mathfrak s$.
Thus equivalence changes $f$ by a coboundary, and every
coboundary comes from such a pair. If the coefficients before
$T_m$ vanish, nonlinear terms start in order $2m$.
The coefficient of $t^m$ is then $\delta_TT_m$.
\end{proof}

\begin{proposition}\label{prop:linear}
The polynomial $T+tf$ is an extended $\mathcal O$-operator
for formal $t$ exactly when
\begin{equation}\label{eq:linear-deformation}
 \delta_Tf=0,\qquad
 \ell_2^T(f,f)=0,\qquad \ell_3(f,f,f)=0.
\end{equation}
An infinitesimal deformation need not extend to a linear
or formal deformation.
\end{proposition}

\begin{proof}
The coefficients of $t,t^2,t^3$ in
\eqref{eq:twistedMC} are, respectively,
$\delta_Tf$, $\frac12\ell_2^T(f,f)$ and
$\frac16\ell_3(f,f,f)$. There are no higher powers.
\end{proof}

\begin{theorem}\label{thm:rigidity}
If $H_T^1=0$, every formal deformation of $T$ is equivalent
to the constant one. Thus $T$ is formally rigid for the fixed
data and \eqref{eq:equivalence}.
\end{theorem}

\begin{proof}
Let $t^mT_m$ be the first term to remove.
By Theorem~\ref{thm:infinitesimals}, $T_m$ is a cocycle.
Since $H_T^1=0$, take $(A_m,B_m)\in\mathfrak s$ with
$T_m=A_mT-TB_m$. The transformation
\[
 T_t\longmapsto e^{-t^mA_m}T_t e^{t^mB_m}
\]
removes the term of order $m$ and keeps all earlier terms.
By \eqref{eq:sym1}--\eqref{eq:sym3}, the exponentials preserve
the data. Repeat in increasing order. The products converge
in the $t$-adic topology: the step at $m$ starts in order $m$.
Their limit gives \eqref{eq:equivalence} with $T'_t=T$.
\end{proof}

We next find when a finite-order deformation extends.

\begin{theorem}\label{thm:obstruction}
Let $T^{(N)}_t=T+\sum_{i=1}^Nt^iT_i$ be a deformation
of order $N\geq1$. Define a $2$-cochain
\begin{align}
 \Ob_{N+1}
 &=\frac12\sum_{\substack{i+j=N+1\\1\leq i,j\leq N}}
                   \ell_2^T(T_i,T_j)\nonumber\\
 &\quad+\frac16
       \sum_{\substack{i+j+k=N+1\\1\leq i,j,k\leq N}}
                   \ell_3(T_i,T_j,T_k).
 \label{eq:obstruction}
\end{align}
Then $\delta_T\Ob_{N+1}=0$ and $\delta_T^*\Ob_{N+1}=0$.
The deformation extends to order $N+1$ exactly when
$[\Ob_{N+1}]=0$ in $\widehat H_T^2$, or equivalently when
its image in $H_T^2$ is zero. The new coefficient must satisfy
\begin{equation}\label{eq:extension-equation}
 \delta_TT_{N+1}=-\Ob_{N+1}.
\end{equation}
If $\widehat H_T^2=0$, every infinitesimal extends to a formal
deformation. The stronger condition $H_T^2=0$ also suffices.
\end{theorem}

\begin{proof}
Write $f_t=\sum_{i=1}^Nt^iT_i$. Its twisted
Maurer--Cartan expression is
\[
 \ell_1^T(f_t)+\tfrac12\ell_2^T(f_t,f_t)
                         +\tfrac16\ell_3(f_t,f_t,f_t).
\]
All coefficients through order $N$ vanish. The next is
\eqref{eq:obstruction}. The $L_\infty[1]$ identities give
\begin{align*}
 &\ell_1^T\!\left(
   \ell_1^T(f_t)+\tfrac12\ell_2^T(f_t,f_t)
                         +\tfrac16\ell_3(f_t,f_t,f_t)\right)\\
 &+\ell_2^T\!\left(f_t,
   \ell_1^T(f_t)+\tfrac12\ell_2^T(f_t,f_t)
                         +\tfrac16\ell_3(f_t,f_t,f_t)\right)\\
 &+\tfrac12\ell_3\!\left(f_t,f_t,
   \ell_1^T(f_t)+\tfrac12\ell_2^T(f_t,f_t)
                         +\tfrac16\ell_3(f_t,f_t,f_t)\right)=0.
\end{align*}
The second and third lines start in orders $N+2$ and $N+3$.
The coefficient of $t^{N+1}$ gives $\ell_1^T(\Ob_{N+1})=0$,
hence $\delta_T\Ob_{N+1}=0$ by \eqref{eq:sign}.
For the cyclic equations, apply Lemma~\ref{lem:cyclic-defect}
with $U=T_t^{(N)}$. The graph defect starts with
$t^{N+1}\Ob_{N+1}$, and $\delta_U^*=\delta_T^*+O(t)$.
The leading coefficient gives $\delta_T^*\Ob_{N+1}=0$.
Thus the obstruction lies in the numerator of
\eqref{eq:H2-corrected}.

Adding $t^{N+1}T_{N+1}$ changes this coefficient by
$\ell_1^T(T_{N+1})=\delta_TT_{N+1}$.
Thus an extension solves \eqref{eq:extension-equation}.
This is possible exactly when the class in $\widehat H_T^2$
vanishes. The injection $\widehat H_T^2\hookrightarrow H_T^2$
uses the same boundaries, so the other criterion is equivalent.
If $\widehat H_T^2=0$, solve at every order.
This gives a formal series; analytic convergence is not required.
\end{proof}

Explicitly, with $T_0=T$ and all indices at most $N$,
\begin{align}
 &(\Ob_{N+1})_{\rm I}(u,v)\nonumber\\
 &=\sum_{\substack{i+j=N+1\\0\leq i,j\leq N}}
 \left([T_i u,T_j v]-T_i(\rho(T_j u)v-\rho(T_j v)u)\right),
 \label{eq:ob-explicit2}\\
 &(\Ob_{N+1})_{\rm II}(u,v,w)\nonumber\\
 &=\sum_{\substack{i+j+k=N+1\\0\leq i,j,k\leq N}}
 \Bigl(\tri{T_i u}{T_j v}{T_k w}
   -T_i\bigl(\mu(T_j v,T_k w)u-\mu(T_j u,T_k w)v
                                +D(T_j u,T_k v)w\bigr)\Bigr).
 \label{eq:ob-explicit3}
\end{align}
The ternary mass terms are linear in $T_t$.
Their order-$N+1$ terms are part of $\delta_TT_{N+1}$,
so they do not occur in the known obstruction.

We now compare deformations of the operator and its induced algebra.

So far, $T$ varies in a fixed double. Its obstruction has
normal coefficients. If both products of a Lie--Yamaguti
algebra $L$ vary, one must deform all four identities.
We give the low-degree complex and its relation to operators.

Let $C^{(2p,2p+1)}(L,L)$ denote the usual paired cochains,
and put
\[
 W(L,L)=\Hom(\wedge^3L,L)
              \oplus\Hom(\wedge^3L\otimes L,L).
\]
Use the adjoint representation for $\delta$ and $\delta^*$.
Starting in degree zero, the deformation complex is
\begin{equation}\label{eq:algebra-deformation-complex}
 \End(L)\xrightarrow{\delta}C^{(2,3)}(L,L)
 \xrightarrow{(\delta^*,\delta)}W(L,L)\oplus C^{(4,5)}(L,L)
 \xrightarrow{(0,\delta)}C^{(6,7)}(L,L)\longrightarrow\cdots.
\end{equation}
This is the cyclic augmentation of \cite{Goswami2025}, with
explicit symmetries for the auxiliary spaces.
Denote its cohomology by $H_{\rm def}^\bullet(L)$. Then
\begin{align}
 H_{\rm def}^1(L)
  &=\frac{\ker\delta^*\cap\ker\delta\,\big|_{C^{(2,3)}}}
               {\delta\End(L)},\label{eq:algebra-H1}\\
 H_{\rm def}^2(L)
  &=\frac{\{(P,Q;U,V)\in W(L,L)\oplus C^{(4,5)}(L,L):
                                      \delta(U,V)=0\}}
               {\{(\delta^*(F,G),\delta(F,G)):(F,G)\in C^{(2,3)}(L,L)\}}.
 \label{eq:algebra-H2}
\end{align}
Thus $H_{\rm def}^1$ is the usual $H^{(2,3)}$.
The obstruction group $H_{\rm def}^2$ includes the cyclic
components. Superscripts denote complex degrees, not
numbers of cochain inputs.

\begin{proposition}\label{prop:full-algebra-obstruction}
Suppose $B_t=B+\sum_{i=1}^Nt^iF_i$ and
$C_t=C+\sum_{i=1}^Nt^iG_i$ satisfy all four Lie--Yamaguti
identities modulo $t^{N+1}$. In the sums below,
$i,j\geq1$ and $i+j=N+1$:
\begin{align}
 P(x,y,z)
  &=\sum_{i+j=N+1}\cyc_{x,y,z}F_i(F_j(x,y),z),
 \label{eq:alg-ob-P}\\
 Q(x,y,z,w)
  &=\sum_{i+j=N+1}\cyc_{x,y,z}G_i(F_j(x,y),z,w),
 \label{eq:alg-ob-Q}\\
 U(x,y,z,w)
  &=\sum_{i+j=N+1}\bigl(G_i(x,y,F_j(z,w))
            -F_i(G_j(x,y,z),w)\nonumber\\
  &\hspace{11em}-F_i(z,G_j(x,y,w))\bigr),
 \label{eq:alg-ob-U}\\
 V(x,y,z,w,a)
  &=\sum_{i+j=N+1}\bigl(G_i(x,y,G_j(z,w,a))
            -G_i(G_j(x,y,z),w,a)\nonumber\\
  &\hspace{5em}-G_i(z,G_j(x,y,w),a)
            -G_i(z,w,G_j(x,y,a))\bigr).
 \label{eq:alg-ob-V}
\end{align}
Then $(P,Q;U,V)$ gives a class in $H_{\rm def}^2(L)$.
An extension by $(F_{N+1},G_{N+1})$ exists exactly when
this class vanishes. Its equation is
\begin{equation}\label{eq:alg-extension}
 (\delta^*,\delta)(F_{N+1},G_{N+1})=-(P,Q;U,V).
\end{equation}
The first nonzero coefficient is a cocycle.
Later coefficients generally satisfy inhomogeneous equations.
\end{proposition}

\begin{proof}
Substitute $B_t+t^{N+1}F_{N+1}$ and $C_t+t^{N+1}G_{N+1}$
in the first cyclic identity. The coefficient of $t^{N+1}$ is
\begin{align*}
 &\cyc_{x,y,z}\bigl(
     B(F_{N+1}(x,y),z)+F_{N+1}(B(x,y),z)
                         +G_{N+1}(x,y,z)\bigr)+P(x,y,z)\\
 &=(\delta^*(F_{N+1},G_{N+1}))_{\rm I}(x,y,z)+P(x,y,z).
\end{align*}
The linear ternary term belongs to $\delta^*$.
The nonlinear terms give $P$.
The second cyclic identity gives the coefficient
\begin{align*}
 &\cyc_{x,y,z}\bigl(
     C(F_{N+1}(x,y),z,w)+G_{N+1}(B(x,y),z,w)\bigr)
                                                   +Q(x,y,z,w)\\
 &=(\delta^*(F_{N+1},G_{N+1}))_{\rm II}(x,y,z,w)+Q(x,y,z,w).
\end{align*}
In the binary derivation identity, the terms with one new
coefficient are
\begin{align*}
 &G_{N+1}(x,y,B(z,w))+C(x,y,F_{N+1}(z,w))\\
 &\quad-F_{N+1}(C(x,y,z),w)-B(G_{N+1}(x,y,z),w)\\
 &\quad-F_{N+1}(z,C(x,y,w))-B(z,G_{N+1}(x,y,w)),
\end{align*}
which give $(\delta(F_{N+1},G_{N+1}))_{\rm I}$.
The other terms give $U$.
In the ternary derivation identity, the new linear terms are
\begin{align*}
 &G_{N+1}(x,y,C(z,w,a))+C(x,y,G_{N+1}(z,w,a))\\
 &\quad-G_{N+1}(C(x,y,z),w,a)-C(G_{N+1}(x,y,z),w,a)\\
 &\quad-G_{N+1}(z,C(x,y,w),a)-C(z,G_{N+1}(x,y,w),a)\\
 &\quad-G_{N+1}(z,w,C(x,y,a))-C(z,w,G_{N+1}(x,y,a)).
\end{align*}
These give $(\delta(F_{N+1},G_{N+1}))_{\rm II}$.
The other terms give $V$. This proves \eqref{eq:alg-extension}.

We check the cocycle condition. Set $\Delta_t=(B_t,C_t)$.
By \eqref{eq:insertI}--\eqref{eq:insertII}, at order $N+1$,
\[
 [\Delta_t,\Delta_t]_{\rm LY}
       =-2t^{N+1}(U,V)+O(t^{N+2}).
\]
The graded Jacobi identity gives
$[\Delta_t,[\Delta_t,\Delta_t]_{\rm LY}]_{\rm LY}=0$.
Its coefficient of $t^{N+1}$ gives $[\Delta,(U,V)]_{\rm LY}=0$,
hence $\delta(U,V)=0$. The truncated equations make all lower
bracket coefficients zero. Only the first nonzero deformation
coefficient must be a cocycle. Thus $(P,Q;U,V)$ lies in the
kernel of $(0,\delta)$. Equation~\eqref{eq:alg-extension} is
solvable exactly when this cocycle is a boundary in
\eqref{eq:algebra-deformation-complex}.
\end{proof}

\begin{proposition}\label{prop:descent-cohomology}
There is a natural linear map
\begin{equation}\label{eq:descent-map}
 \Theta_T:H_T^1\longrightarrow H_{\rm def}^1(V_T),
       \qquad [f]\longmapsto[(b_f,c_f)],
\end{equation}
where
\begin{align}
 b_f(u,v)&=\rho(fu)v-\rho(fv)u,\label{eq:descent-b}\\
 c_f(u,v,w)&=\mu(fv,Tw)u+\mu(Tv,fw)u
             -\mu(fu,Tw)v-\mu(Tu,fw)v\nonumber\\
 &\quad+D(fu,Tv)w+D(Tu,fv)w.
 \label{eq:descent-c}
\end{align}
The pair $(b_f,c_f)$ satisfies both cyclic and both paired
cocycle conditions.
\end{proposition}

\begin{proof}
If $\delta_Tf=0$, Theorem~\ref{thm:differential} makes $T+tf$
an extended $\mathcal O$-operator over $\K[t]/(t^2)$.
Expanding \eqref{eq:desc2}--\eqref{eq:desc3} gives the products
$[u,v]_T+tb_f(u,v)$ and $\Tri{u}{v}{w}+tc_f(u,v,w)$.
All four identities hold modulo $t^2$.
Their linear terms give
$(\delta_{\rm ad}^*,\delta_{\rm ad})(b_f,c_f)=0$.

We check that the map is well defined.
Take $(A,B)\in\mathfrak s$ and $f=AT-TB$.
Since $B\rho(x)-\rho(x)B=\rho(Ax)$,
\begin{align*}
 b_f(u,v)
 &=\rho((AT-TB)u)v-\rho((AT-TB)v)u\\
 &=B[u,v]_T-[Bu,v]_T-[u,Bv]_T.
\end{align*}
The identity for $\mu$, together with $AS=SB$, gives
\[
 Bc(u,v,w)=c(Bu,v,w)+c(u,Bv,w)+c(u,v,Bw).
\]
Apply $B$ to $\mu(Sv,Sw)u$. Move $B$ through $\mu$ and
replace $AS$ by $SB$. The same calculation for the other
terms of $\Tri{u}{v}{w}$ gives
\begin{align*}
 c_f(u,v,w)
  =B\Tri{u}{v}{w}
     -\Tri{Bu}{v}{w}-\Tri{u}{Bv}{w}-\Tri{u}{v}{Bw}.
\end{align*}
These expressions are $-\delta_{\rm ad}B$.
Thus operator coboundaries give algebra coboundaries,
so \eqref{eq:descent-map} is well defined.
Equations \eqref{eq:descent-b}--\eqref{eq:descent-c} give linearity.
The map need not be injective or surjective.
\end{proof}

For a formal operator family, the induced products satisfy
all four obstruction equations at each order.
By Proposition~\ref{prop:descent-cohomology}, checking only
the two derivation identities need not suffice for a proposed
deformation of $V_T$.

The examples below give an obstruction, rigid operators
and a nontrivial formal deformation.

\begin{example}\label{ex:obstructed}
Take the two-dimensional Lie algebra $\g$ with $[e_1,e_2]=e_2$,
$\tri{x}{y}{z}=[[x,y],z]$, its adjoint representation,
$S=\id$ and $\kappa=0$. At $T=0$, the induced products and
normal representation are zero. Thus $\delta_T=0$ in every
positive degree, and $f=\id$ is infinitesimal.
Its second-order obstruction is
\[
 (\Ob_2)_{\rm I}(u,v)
       =[u,v]-([u,v]-[v,u])=-[u,v],
 \qquad(\Ob_2)_{\rm II}=0.
\]
Here $(\Ob_2)_{\rm I}(e_1,e_2)=-e_2\ne0$.
The image of $\delta_T:C_T^1\to C_T^2$ is zero, so the
obstruction class is nonzero. This infinitesimal cannot
extend to order two.
\end{example}

\begin{theorem}\label{thm:scalar}
Use the adjoint representation of a Lie--Yamaguti algebra $\g$.
Set $S=\id$, $\kappa=-c_0^2$ with $c_0\ne0$, and $T=c_0\id$.
Then $T$ is an extended $\mathcal O$-operator, with
\begin{equation}\label{eq:scalar-products}
 [u,v]_T=2c_0[u,v],\qquad
 \Tri{u}{v}{w}=4c_0^2\tri{u}{v}{w}.
\end{equation}
Its normal representation is trivial, and
\begin{equation}\label{eq:scalar-H1}
 H_T^1\cong
 \Hom\!\left(\g/
    \spanof\{[u,v],\tri{u}{v}{w}:u,v,w\in\g\},\g\right).
\end{equation}
If the span is all of $\g$, then $T$ is formally rigid.
\end{theorem}

\begin{proof}
The binary equation is
$c_0^2[u,v]-2c_0^2[u,v]=-c_0^2[u,v]$.
Each of the three adjoint terms in \eqref{eq:desc3} is
$c_0^2\tri{u}{v}{w}$ by \eqref{eq:ad} and skew-symmetry
in the first two inputs. The mass term adds a fourth copy.
The maps $T+c_0\id=2c_0\id$ and $T-c_0\id=0$ preserve
\eqref{eq:scalar-products}. Thus Proposition~\ref{prop:roots}
also proves the operator equations.

In the split model of Proposition~\ref{prop:models}, the graph
is $\g\oplus0$ and the quotient is the second factor.
The first factor acts trivially on the second.
Thus $\rho_T=\mu_T=D_T=0$, and
\[
 \delta_T f=
 \bigl((u,v)\mapsto-2c_0f([u,v]),
       (u,v,w)\mapsto-4c_0^2f(\tri{u}{v}{w})\bigr).
\]
Since $S=\id$, each symmetry pair has $A=B$, so $AT-TB=0$.
All degree-one coboundaries vanish. The kernel above gives
\eqref{eq:scalar-H1}. Theorem~\ref{thm:rigidity} proves rigidity.
\end{proof}

\begin{example}\label{ex:nonLie}
The theorem also applies when the binary product is not a Lie
bracket. Take $\mathfrak{sl}_3(\mathbb R)=\h\oplus\mathfrak m$,
where $\h$ is diagonal and traceless and $\mathfrak m$ is spanned
by the six off-diagonal matrix units. On $\mathfrak m$, set
\[
 [x,y]_{\mathfrak m}=\pr_{\mathfrak m}(xy-yx),
 \qquad
 \tri{x}{y}{z}=[\pr_{\h}(xy-yx),z]_{\mathfrak{sl}_3}.
\]
This reductive decomposition gives a Lie--Yamaguti algebra.
For $x=E_{12}$, $y=E_{21}$ and $z=E_{23}$, the binary Jacobi
sum is $E_{23}$ and the ternary cyclic sum is $-E_{23}$.
Neither is zero. For distinct $i,j,k$,
$[E_{ik},E_{kj}]_{\mathfrak m}=E_{ij}$.
Thus binary products span $\mathfrak m$.
By Theorem~\ref{thm:scalar}, $T=c_0\id_{\mathfrak m}$ is rigid.
\end{example}

\begin{example}\label{ex:family}
Let $\g=\spanof\{e_1,e_2\}$ with $[e_1,e_2]=e_2$
and $\tri{x}{y}{z}=[[x,y],z]$. Use the adjoint
representation, $S=\id$, $\kappa=-1$, and set
\begin{equation}\label{eq:family}
 T_a e_1=a e_1,\qquad T_a e_2=e_2,\qquad a\in\mathbb R.
\end{equation}
For the only independent binary pair,
\[
 [T_ae_1,T_ae_2]-T_a([T_ae_1,e_2]+[e_1,T_ae_2])
       =a e_2-(a+1)e_2=-e_2.
\]
This proves \eqref{eq:operator2}.
Proposition~\ref{prop:Lie} gives the ternary equation.
The induced products are
\begin{equation}\label{eq:family-products}
 [e_1,e_2]_{T_a}=(a+1)e_2,\qquad
 \Tri{e_1}{e_2}{e_1}=-(a+1)^2e_2,\qquad
 \Tri{e_1}{e_2}{e_2}=0.
\end{equation}
At $a=-1$, both induced products vanish.
For $a\ne-1$, both are nonzero.

Replace $a$ by $a+t$. This gives an unobstructed deformation
with infinitesimal $f(e_1)=e_1$, $f(e_2)=0$.
Its class in $H_{T_a}^1$ is nonzero.
Every derivation of the binary bracket has $Ae_1=s e_2$ and
$Ae_2=r e_2$. It is also a derivation of the ternary product.
Since $S=\id$, symmetry pairs have $B=A$. Hence
\[
 (AT_a-T_aA)e_1=(a-1)s e_2,\qquad
 (AT_a-T_aA)e_2=0.
\]
No such coboundary equals $f$.
Thus the formal deformation is nontrivial, with a nonzero
class in $H_{T_a}^1$.
\end{example}

{\bf Declaration of interests}

The author declares no conflict of interest.

{\bf Data availability}

No data were used in this study.

{\bf Funding}

This work received no funding.



\end{document}